\documentclass[reqno]{amsart}

\usepackage[all]{xy}
\usepackage{graphicx}

\usepackage{amssymb}

\usepackage{comment}
\usepackage{amsmath}
\usepackage{mathrsfs}
\usepackage{epsfig}
\usepackage{amscd}
\usepackage{graphicx, color}

\def\E{\ifmmode{\mathbb E}\else{$\mathbb E$}\fi} 
\def\N{\ifmmode{\mathbb N}\else{$\mathbb N$}\fi} 
\def\R{\ifmmode{\mathbb R}\else{$\mathbb R$}\fi} 
\def\Q{\ifmmode{\mathbb Q}\else{$\mathbb Q$}\fi} 
\def\C{\ifmmode{\mathbb C}\else{$\mathbb C$}\fi} 
\def\H{\ifmmode{\mathbb H}\else{$\mathbb H$}\fi} 
\def\Z{\ifmmode{\mathbb Z}\else{$\mathbb Z$}\fi} 
\def\P{\ifmmode{\mathbb P}\else{$\mathbb P$}\fi} 
\def\T{\ifmmode{\mathbb T}\else{$\mathbb T$}\fi} 
\def\SS{\ifmmode{\mathbb S}\else{$\mathbb S$}\fi} 
\def\DD{\ifmmode{\mathbb D}\else{$\mathbb D$}\fi} 

\newcommand{\del}{\partial}
\newcommand{\Cont}{{\operatorname{Cont}}}

\newcommand{\ben}{\begin{enumerate}}
\newcommand{\een}{\end{enumerate}}
\newcommand{\be}{\begin{equation}}
\newcommand{\ee}{\end{equation}}
\newcommand{\bea}{\begin{eqnarray}}
\newcommand{\eea}{\end{eqnarray}}
\newcommand{\beastar}{\begin{eqnarray*}}
\newcommand{\eeastar}{\end{eqnarray*}}

\theoremstyle{theorem}
\newtheorem{thm}{Theorem}[section]
\newtheorem{cor}[thm]{Corollary}
\newtheorem{lem}[thm]{Lemma}

\newtheorem{prop}[thm]{Proposition}

\theoremstyle{definition}
\newtheorem{defn}[thm]{Definition}
\newtheorem{rem}[thm]{Remark}

\newtheorem{exm}[thm]{Example}

\newtheorem*{thm*}{Theorem}

\numberwithin{equation}{section}

\def\R{{\mathbb R}}

\def\E{{\mathbb E}}
\def\Z{{\mathbb Z}}
\def\C{{\mathbb C}}
\def\R{{\mathbb R}}
\def\P{{\mathbb P}}

\def\N{{\mathbb N}}

\def\11{{\mathbb I}}

\def\delbar{{\overline \partial}}

\def\id{\text{\rm id}}

\def\V{\mathbb{V}}
\def\C{\mathbb{C}}
\def\Z{\mathbb{Z}}

\def\T{\mathbb{T}}

\def\Q{\mathbb{Q}}

\def\E{\ifmmode{\mathbb E}\else{$\mathbb E$}\fi} 
\def\N{\ifmmode{\mathbb N}\else{$\mathbb N$}\fi} 
\def\R{\ifmmode{\mathbb R}\else{$\mathbb R$}\fi} 
\def\Q{\ifmmode{\mathbb Q}\else{$\mathbb Q$}\fi} 
\def\C{\ifmmode{\mathbb C}\else{$\mathbb C$}\fi} 
\def\H{\ifmmode{\mathbb H}\else{$\mathbb H$}\fi} 
\def\Z{\ifmmode{\mathbb Z}\else{$\mathbb Z$}\fi} 
\def\P{\ifmmode{\mathbb P}\else{$\mathbb P$}\fi} 
\def\SS{\ifmmode{\mathbb S}\else{$\mathbb S$}\fi} 
\def\DD{\ifmmode{\mathbb D}\else{$\mathbb D$}\fi} 

\def\R{{\mathbb R}}

\def\E{{\mathbb E}}
\def\Z{{\mathbb Z}}
\def\C{{\mathbb C}}
\def\R{{\mathbb R}}

\def\N{{\mathbb N}}

\def\delbar{{\overline \partial}}

\def\CE{{\mathcal E}}
\def\CF{{\mathcal F}}

\def\CI{{\mathcal I}}
\def\CJ{{\mathcal J}}

\def\CL{{\mathcal L}}
\def\CM{{\mathcal M}}
\def\CN{{\mathcal N}}

\def\CP{{\mathcal P}}

\def\CP{{\mathcal P}}

\def\darr#1{\raise1.5ex\hbox{$\leftrightarrow$}
\mkern-16.5mu #1}

\def\roughly#1{\raise.3ex\hbox{$#1$\kern-.75em
\lower1ex\hbox{$\sim$}}}

\def\opname#1{\mathop{\kern0pt{\rm #1}}\nolimits}

\def\dim{\opname{dim}}
\def\vol{\opname{vol}}

\def\rank{\opname{rank}}

\def\supp{\operatorname{supp}}

\def\coker{\operatorname{Coker}}
\def\span{\operatorname{span}}

\def\Cont{\operatorname{Cont}}

\def\Spec{\operatorname{Spec}}

\def\Diff{\operatorname{Diff}}
\def\Image{\operatorname{Image}}
\def\ev{\operatorname{ev}}
\def\Ev{\operatorname{Ev}}

\def\Symp{\operatorname{Symp}}

\def\Gr{\operatorname{Gr}}

\DeclareFontFamily{U}{MnSymbolC}{}
\DeclareSymbolFont{MnSyC}{U}{MnSymbolC}{m}{n}
\DeclareFontShape{U}{MnSymbolC}{m}{n}{
    <-6>  MnSymbolC5
   <6-7>  MnSymbolC6
   <7-8>  MnSymbolC7
   <8-9>  MnSymbolC8
   <9-10> MnSymbolC9
  <10-12> MnSymbolC10
  <12->   MnSymbolC12}{}
\DeclareMathSymbol{\intprod}{\mathbin}{MnSyC}{'270}

\begin{document}

\quad \vskip1.375truein

\title[Generic Reeb foliations]{Leafwise de Rham cohomology of generic Reeb foliations}

\author{Yong-Geun Oh}
\address{Center for Geometry and Physics, Institute for Basic Science (IBS),
 79 Jigok-ro 127beon-gil, Nam-gu, Pohang, Gyeongbuk, KOREA 37673
\& POSTECH, Gyeongbuk, Korea}
\email{yongoh1@postech.ac.kr}

\thanks{The present research is supported by the IBS project \# IBS-R003-D1}

\date{April 2025, revised in July 2026}

\begin{abstract} In this paper, we prove that there exists a residual subset of contact forms $\lambda$
(if any) on any compact connected orientable manifold $M$ for which the foliation de Rham cohomology of the associated Reeb foliation has $ H^0(\CF_\lambda) \cong \R$.
We also prove the same triviality for a generic choice of contact forms with \emph{fixed}
contact structure $\xi$.  This vanishing result of $H^0(\CF_\lambda)$
is also equivalent to the statement that
the Lie algebra of the group of \emph{strict} contactomorphisms is isomorphic to
the span of Reeb vector fields, and so isomorphic to the 1 dimensional abelian  Lie algebra $\R$.
On the other hand, we derive the rank of $H^1(\CF_\lambda)$ 
is infinite \emph{whenever $\lambda$ admits a closed Reeb oribt, i.e., whenever Weinstein's 
conjecture holds}. In particular we prove that $H^1(\CF_\lambda)$ is infinite dimensional 
for all contact form $\lambda$ in dimension 3, thanks to Taubes' proof of 3-dimensional Weinstein's conjecture.
\end{abstract}
 

\maketitle

\tableofcontents

\section{Introduction}

Not every odd dimensional orientable smooth compact manifold admits 
a maximally nondegenerate one-form called  a contact form.
The odd dimensional Stiefel-Whitney classes have to vanish for such a manifold \cite{gray}.
(All orientable 3 dimensional manifold admits a contact structure 
\cite{lutz}, \cite{martinet}, \cite{thurston-winkelnkemper}, though.)
We call a manifold \emph{contactable} if it admits a contact form. 
  
Let $M$ be an odd dimensional compact connected contactable manifold without boundary,
and denote by $\mathfrak{C}(M) \subset \Omega^1(M)$
the set of contact forms (not specified by a given contact structure), i.e., 
\be\label{eq:contact-form-set}
\mathfrak{C}(M) = \{\lambda \in \Omega^1(M) \mid d\lambda \, 
\text{\rm is nondegenerate on $\ker \lambda$}\}
\ee
For any $\lambda \in \mathfrak{C}(M)$,  the distribution $\xi = \ker \lambda$ defines a 
contact structure that is coorientable. Conversely any coorientable
contact manifold $(M,\xi)$ is defined in this way. 
When we are given a contact manifold $(M,\xi)$, we consider the subset of $\mathfrak{C}(M)$
defined by
\be\label{eq:contactform-xi}
\mathfrak{C}(M,\xi): = \{ \lambda \in \mathfrak{C}(M) \mid \ker \lambda = \xi\}.
\ee
For the convenience of presentation, we will call $\mathfrak{C}(M)$ the \emph{big phase space}
and $\mathfrak{C}(M,\xi)$ the \emph{small phase space} adopting the terminologies from \cite{do-oh:reduction}.

A choice of contact form $\lambda$ determines a unique (pointwise) decomposition
\be\label{eq:T*M-decomposition}
T^*M =  \{V^\pi \intprod d\lambda \mid V^\pi \in \xi\} \oplus \R \{\lambda\}
\ee
which induces the splitting
\be\label{eq:TM-decomposition}
TM  = (\R\{\lambda\})^\perp \oplus
 \{V^\pi \intprod d\lambda \mid V^\pi \in \xi\}^\perp = \ker \lambda \oplus \R\{R_\lambda\} 
 \ee
where $R_\lambda$ is the Reeb vector field uniquely determined by 
\be\label{eq:defining-Reeb}
R_\lambda \rfloor \lambda = 1, \quad R_\lambda \rfloor d\lambda = 0.
\ee
(Here we denote by $(\cdot)^\perp$ the null-space of $(\cdot)$.)
Another natural consequence of the presence of a contact form $\lambda$ is 
a volume form $\mu_\lambda = \lambda \wedge (d\lambda)^n$ and its associated
measure on $M$, which we also denote by $\mu_\lambda$ with a slight
abuse of notation. We call it the \emph{contact Liouville measure} associated to $\lambda$.

We now borrow the following definition from \cite{do-oh:reduction} for the convenience of further exposition.

\begin{defn}[Contact pair] We call  a \emph{contact pair} any pair $(\lambda,\psi)$
of  a contact form  $\lambda$ and a diffeomorphism $\psi$ of $M$ that satisfies
\be\label{eq:contact-pair}
 d\psi (\xi) \subset \xi
 \ee
 for the kernel $\xi = \ker \lambda$ of $\lambda$.
We say it is positive (resp. negative) if $\psi^*\lambda = f \lambda$ with $f > 0$ (resp. $f < 0$).
\end{defn}
For any positive contact pair $(\lambda,\psi)$, we can write $\psi^*\lambda = e^g \lambda$
by putting.  We call $f$ the \emph{conformal factor}of $(\lambda,\psi)$, and
its logarithm  $g : = \log f$ the  \emph{conformal exponent}. To make the dependence on $(\lambda,\psi)$
explicit, we write
$$
g = g_{(\psi;\lambda)}.
$$
When $\lambda$ is fixed, we call it just the $\lambda$-conformal exponent of $\psi$ and just write $g_\psi$
omitting $\lambda$ from notation. (This is the reason why we write $\psi$ first by switching the locations 
of $(\lambda,\psi)$ for the notation $g_{(\psi;\lambda)}$.)

\subsection{Strict contactomorphisms}

A strict contactomorphism $\psi$ satisfies the equation 
\be\label{eq:strict-contact} 
\psi^*\lambda = \lambda.
\ee
It forms a Lie group \cite{casals-spacil}, \cite{oh-savelyev:strict-contact}, which we denote by 
$\Cont^{\text{\rm st}}(M,\lambda)$. We denote by
$$
\mathfrak{X}^{\text{\rm st}}(M,\lambda)
$$
the associated Lie algebra. The defining equation \eqref{eq:strict-contact} not only requires 
a strict contactomorphism $\psi$ to satisfy
\be\label{eq:contact}
d\psi(\xi) = \xi
\ee
for the associated contact distribution $\xi: =\ker \lambda$, which is the definition of 
a general contactomorphism, but also to preserve the associated
Reeb vector field $\psi_*R_\lambda = R_\lambda$. This `overdeterminedness' is the main reason why
a strict contact diffeomorphism is hard to construct and has been scarce beyond the universally 
present Reeb flows $\phi_{R_\lambda}^t$.

\begin{defn}[Strict contact pair] Let $\lambda$ be a contact form and $\psi \in \Diff(M)$.
We call a pair $(\lambda,\psi)$  a \emph{strict contact pair} if
it satisfies the equation $\psi^*\lambda = \lambda$.
\end{defn} 
For any strict contact pair $(\lambda,\psi)$, we have $g_{(\lambda;\psi)} \equiv 0$ and hence
the discriminant is highly degenerate for any strict contact pair in that  its zero-set  
becomes a full space $M$ instead of a hypersurface.
\begin{rem}
Recall that for any non-strict contact pair 
the associated conformal exponent cannot be constan
and hence its zero set cannot be a full space.
This suggests that appearance of a strict contactomorphism $\psi$ is not a generic
phenomenon under the choice of $\lambda$.  In a companion paper \cite{oh-savelyev:strict-contact} jointed with Y. Savelyev,
we will show that this is indeed the case.
\end{rem}

Now we consider the Lie algebra version of the above discussion. 
Similarly as the defining equation of the Reeb vector field \eqref{eq:defining-Reeb}, which corresponds to
$H = -1$, we can associate the contact vector field $X$
uniquely determined by the equation 
 \be\label{eq:XlambdaH}
\begin{cases}
X \intprod \lambda = -H, \\
X \intprod d\lambda = dH - R_\lambda[H] \lambda.
\end{cases}
\ee
Here we denote the Lie derivative $R_\lambda[H] = \CL_{R_\lambda}H$ as usual.
We denote this $X$ by $X_H$ and call the \emph{contact Hamiltonian vector field} of $H$, and denote its flow 
(with $\psi_H^0 = \id$) by
$\psi_H^t$ in general for a (time-dependent) Hamiltonian $H = H(t,x)$.
The bracket on $C^\infty(M,\R)$ defined by
\be\label{eq:bracket-HG}
\{H,G\}: = \lambda([X_H,X_G]),
\ee
often called the \emph{Jacobi-Lagrange bracket},  defines a Lie-algebra structure on $C^\infty(M,\R)$
and induces the canonical  Lie algebra isomorphism defined by
$$
\mathfrak{X}(M,\lambda) \cong C^\infty(M,\R); \quad X \mapsto  - \lambda(X).
$$
(See \cite{arnold:book}, \cite{boyer}, \cite{LOTV}, \cite{dMV} for the definition of Jacobi-Lagrange bracket
under various sign conventions. We follow that of \cite[Proposition 9]{dMV}.)

We consider a pair $(\lambda, H)$ of a contact form $\lambda$ and an autonomous function $H$
such that the whole flow $\psi_{(H;\lambda)}^t$, which we often just write $\psi_H^t$ suppressing $\lambda$-dependence
for the simplicity of notations,
 is $\lambda$-strict contact, or equivalently that  the pair satisfies
\be\label{eq:pair-lambdaH}
g_{(\psi_H^t;\lambda)} \equiv 0 \quad \forall t \in \R.
\ee
By differentiating this equation in time $t$, we  obtain the following 
description of $\mathfrak{X}^{\text{\rm st}}(M,\lambda)$, which is well-known to the
experts. (See \cite{casals-spacil} for example.)
\begin{lem}\label{lem:strict-contact-H} 
The Hamiltonian flow $\psi_H^t$ is a strict contact flow, i.e., \eqref{eq:pair-lambdaH} holds
if and only if $R_\lambda[H] = 0$. In particular,
\be\label{eq:cont-st}
\mathfrak{X}^{\text{\rm st}}(M,\lambda) = \{X_H \mid R_\lambda[H] = 0, \ H \in C^\infty(M)\}.
\ee
\end{lem}
(The first statement of this lemma is an immediate consequence of the explicit formula \eqref{eq:dgdt}
for the time derivative of $g_{(\psi_H^t;\lambda)}$.)

\subsection{Statement of main theorems and cohomological equation}

Notice that any function $H$ satisfying $R_\lambda[H| = 0$ is constant along the Reeb trajectories
and so every such a function $H$ on $M$ is projectible to the leaf space
$\CN_\lambda$ consisting of Reeb trajectories. We would like to mention that $\CN_\lambda$
is rarely Hausdorff, which turns out to be a generic phenomenon by the main theorem of
the present paper, Theorem \ref{thm:nonprojectable-intro} below.

The deformation problem of strict contactomorphisms for a given contact $\lambda$ is 
closely tied to the dynamics of the Reeb flows which is also 
connected to the topology and dynamics of this \emph{Reeb foliation} denoted by $\CF_\lambda$. 

The following is the first main theorem of the present paper.  The characterization of the 
Lie algebra $\mathfrak{X}^{\text{\rm st}}(M,\lambda)$ given in Lemma \ref{lem:strict-contact-H}
 motivates us to carefully
study the cohomological equation
\be\label{eq:Rlambda[f]=g}
R_\lambda[f] = u.
\ee
(We would like mention that such a \emph{cohomological equation} $X[f] = u$ 
in Dynamical Systems has been much studied
in relation to the \emph{area-preserving dynamics}
on surfaces. (See \cite{forni97,forni02,forni21,forni22},
 for example. See also \cite{arnold:geometrical} in general dimension, where the same equation 
 is called the \emph{homological  equation} instead.)

\begin{thm}\label{thm:nonprojectable-intro} 
Assume $M$ is connected. There exists a residual subset 
 $$
 \mathfrak{C}^{\text{\rm np}} (M) \subset \mathfrak{C}(M)
 $$
consisting of contact forms $\lambda$ such that the only solutions for the
(co)homological equation $R_\lambda[f] = 0$ are constant functions.
\end{thm}
We refer readers to Section \ref{sec:wrap-up},  especially \eqref{eq:CMnp},
for the precise definition of the residual subset
$ \mathfrak{C}^{\text{\rm np}} (M)$. We will call any element of the residual subset \emph{non-projectible}
(to the leaf space). For such a contact form, only constant functions are
projectible to the leaf space as a single-valued function among smooth functions $H$ on $M$.

An immediate corollary is the following triviality theorem of the Lie algebra 
$\mathfrak{X}^{\text{\rm st}}(M,\lambda)$ of strict contactomorphisms.
(We refer readers to \cite{oh-savelyev:strict-contact} for the lifting of this result to the 
description of  Lie group $\Cont^{\text{\rm st}}(M,\lambda)$.)

\begin{cor} Let $\lambda \in \mathfrak{C}^{\text{\rm np}}(M)$. Then
$$
\mathfrak{X}^{\text{\rm st}}(M,\lambda) = \R \langle R_\lambda\rangle  \cong \R.
$$
\end{cor}

We also have the following \emph{small phase space} counterpart of Theorem \ref{thm:nonprojectable-intro}.

\begin{thm}\label{thm:nonprojectable-small-intro} Let $(M,\xi)$ be a coorientable contact manifold. Then
there exists a residual subset
 $$
 \mathfrak{C}^{\text{\rm np}} (M,\xi) \subset \mathfrak{C}(M,\xi)
 $$
consisting of $\lambda$ such that $\ker \lambda = \xi$ and that
$$
\mathfrak{X}^{\text{\rm st}}(M,\lambda) = \{c R_{\lambda} \mid c \in \R\}.
$$
\end{thm}

We call any such contact form $\lambda \in  \mathfrak{C}^{\text{\rm np}} (M,\xi)$
\emph{non-projectible in small phase space}. We refer readers to Section \ref{sec:wrapup-small}
for the precise definition of $ \mathfrak{C}^{\text{\rm np}} (M,\xi)$.

\begin{rem}
\begin{enumerate} 
\item Recall the fact the Reeb vector field $R_\lambda$ is nothing but the Hamiltonian vector field $X_{(-1)}$ associated
to the constant function $-1$ \cite{dMV}, \cite{oh:contacton-Legendrian-bdy}.
We have $\lambda([R_\lambda,X_H]) = \{-1,H\}$ by the definition the aforementioned Jacobi-Lagrange bracket $\{\cdot,\cdot \}$. 
On the other hand, a direct calculation gives rise to 
$$
\lambda([R_\lambda,X_H]) = -  d\lambda(R_\lambda, X_H) + R_\lambda[\lambda(X_H)] - X_H[\lambda(R_\lambda)]
= -R_\lambda[H].
$$
Therefore $R_\lambda[H] = 0$
is equivalent to the vanishing of the Jacobi-Lagrange bracket $\{1,H\} = 0 $.
\item The notion of completely integrable contact 
Hamiltonian systems has been discussed in the literature.
More recently, in  relation to a new construction of Sasaki-Einstein metrics 
\cite{boyer}, \cite{visinescu}, in the dynamics of the billiard system 
\cite{khesin-tabachinikov} and even more recently in contact topology
\cite{sun-uljar-vargol}. The definition of contact integrable systems is 
given by requiring the existence of $n$ independent functions $\{f_1, \cdots, f_n\}$
in involution as in the symplectic case, i.e., $\{f_i,f_j\} = 0$
with respect to the \emph{Jacobi-Lagrange bracket} $\{\cdot, \cdot\}$ 
but \emph{with the additional constraint $\{1,f_i\} = 0$ for all $i = 1, \cdots, n$}. 
By the result of the present paper, this latter
requirement turns out to be a very strong restriction which makes the given 
definition of the \emph{contact integrable system} hypothetical \emph{for a generic 
contact form $\lambda$}. It also makes the underlying Reeb flow define a circle action
so that the underlying contact manifold is essentially a prequantization of integral symplectic manifold
(or integral symplectic orbifold), and
the given system is essentially a lifting of a symplectic completely
integrable system defined on this integral symplectic orbifold. 
\end{enumerate}
\end{rem}

\subsection{Generic property of the foliation de Rham cohomology}

The above study of the functional equation $R_\lambda[f] = u$ has the cohomological formulation
whose explanation is now in order.

Each contact form $\lambda$ gives rise to a two-term leafwise de Rham complex 
\be\label{eq:E-deRham-complex}
0 \to \Omega^0(\CF_\lambda) \to \Omega^1(\CF_\lambda) \to 0
\ee
associated to the one dimensional foliation $\CF_\lambda$ of the $\lambda$-Reeb trajectories. 
The only possibly non-zero cohomology groups are
$$
H^0(\CF_\lambda), \quad H^1(\CF_\lambda).
$$
Note that 
\be\label{eq:complex}
\Omega^0(\CF_\lambda) = C^\infty(M), \quad 
\Omega^1(\CF_\lambda) = \{u\, \lambda \mid u \in C^\infty(M)\}.
\ee
The following explicit description of the differential $d_\CF$ is  important for our study of 
the cohomological equation
\be\label{eq:Rlambdaf=u}
R_\lambda[f] = u.
\ee
\begin{prop}  Consider the leafwise differential $d_\CF: \Omega^0(\CF_\lambda) \to \Omega^1(\CF_\lambda)$. Then
\beastar
Z^0(\CF_\lambda) & = & \ker d_\CF =  \{ f\,  \mid R_\lambda[f] = 0\}, \\
B^1(\CF_\lambda) & = & \{u \, \lambda \mid u = R_\lambda[f]\, \text{ \rm for some } \, f \in C^\infty\},\\
Z^1(\CF_\lambda) & = & \{u \, \lambda \mid u \in C^\infty(M)\} \cong \coker d_\CF. 
\eeastar
In particular, we have
\bea\label{eq:H0H1-FF}
H^0(\CF_\lambda) & = & \{f \in C^\infty(M) \mid R_\lambda[f] = 0 \}, \nonumber\\
H^1(\CF_\lambda) & \cong  &  
\frac{\{u \lambda \mid u \in C^\infty(M)\}}{\{R_\lambda[u] \lambda \mid u \in C^\infty(M)]\}}.
\eea
\end{prop}
An immediate corollary of this proposition is the following. We write
\be\label{eq:meanzero-set}
C_{(0;\lambda)}^\infty(M,\R): = \left\{ f \in C^\infty(M,\R) \, \Big| \,  \int_M f\, d\mu_\lambda = 0 \right\}
\ee
\begin{cor}  The following holds:
\begin{enumerate} 
\item $\ker R_\lambda = \{0\}$ on $C_{(0;\lambda)}^\infty(M,\R)$ if and only if $H^0(\CF_\lambda) \cong \R$.
\item $R_\lambda[f] = u$ has a solution for all $u \in C_{(0;\lambda)}^\infty(M,\R)$
if and only if  $H^1(\CF_\lambda) \cong \R$.
\end{enumerate}
\end{cor}

Combining the above discussion with Theorem \ref{thm:nonprojectable-intro}, we can rephrase 
it as the following generic triviality theorem on $H^0(\CF_\lambda)$.
\begin{thm}\label{thm:triviality-theorem} Let $\lambda$ be any non-projectible contact form on $M$. Then 
$$
\dim H^0(\CF_\lambda,M)  = 1.
$$
In particular this holds for any contact form 
$\lambda$ in the residual subset $\mathfrak{C}^{\text{\rm np}}(M) \subset \mathfrak C(M)$.
\end{thm}
\begin{exm}\label{exm:sphere} Consider the standard contact form $\lambda_{\text{\rm st}}$
of $S^{2k+1} \subset \C^{k}$ which is 
a contact manifold as the prequantization of $\C P^n$. It follows from the definition of 
$H^0(\CF_{\lambda_{\text{\rm st}}};S^{2k+1})$
$$
H^0(\CF_{\lambda_{\text{\rm st}}}) = \ker R_{\lambda_{\text{\rm st}}} \cong C^\infty(\C P^n,\R).
$$
Our theorem shows that this is a non-generic phenomenon and the cohomology groups will
become trivial, or more precisely, is  isomorphic to $\R$, 
as in Theorem \ref{thm:nonprojectable-intro} by a $C^\infty$-small perturbation 
of $\lambda_{\text{\rm st}}$.
\end{exm}

On the other hand, we state the following general result on $H^1(\CF_\lambda)$.

\begin{thm}\label{thm:do-intro} Let $(M, \xi)$ be a compact contact manifold and $\lambda$ be its contact form.
Suppose that $\lambda$ admits a closed Reeb orbit. Then $H^1(\CF_\lambda)$ is infinite dimensional.
\end{thm}

\begin{rem} This result is pointed out by Hyun-Seok Do \cite{do-example} during our discussion
on the previous version of the present paper,
which results in correcting the author's  incorrect previous claim that $H^1(\CF_\lambda)$ is also
one-dimensional  for a generic choice of $\lambda$.
In Subsection \ref{subsec:hyunseok}, we will present the proof of the theorem we learned from him.
\end{rem}

By combining this theorem with Taubes' proof \cite{taubes} of Weinstein conjecture on 
3-dimensional contact manifolds, we obtain the following corollary.

\begin{cor} Let $(M,\xi)$ be any compact contact 3-manifold. Then 
$$
\dim H^1(\CF_\lambda,M) = \infty
$$
for all contact forms $\lambda$ of $\xi$. More generally, the same also holds for all higher dimensions, 
\emph{whenever Weinstein's conjecture holds.}
\end{cor}

\subsection{Overall strategy of the proof }
\label{subsec:strategy}

In this subsection, we outline the overall strategy of our proof of the main theorems.

What we need to prove is that any function $H$ satisfying $R_\lambda[H] = 0$ must be constant for any
contact form in the to-be-identified residual subset
$\mathfrak{C}^{\text{\rm np}}(M) \subset \mathfrak{C}(M)$
(resp. $\mathfrak{C}^{\text{\rm np}}(M,\xi) \subset \mathfrak{C}(M,\xi)$).

To prove such a generic statement for $\mathfrak{C}(M)$ 
it is natural to attempt to apply the Sard-Smale theorem in a suitable moduli problem 
associated to each pair $(\lambda, H)$ regarding the contact forms as the parameter space.
In the first try, one might want to directly consider the map $(\lambda,H) \mapsto R_\lambda[H]$,
which however turns out to be of no use, e.g., because the operator $R_\lambda$ is not an 
elliptic operator. 
 
The following observation  is one of  the key 
ingredients used in our proof of the main results of the present paper.

\begin{rem}[Characteristic equation]\label{rem:characteristic}
We would like to point out that the (elliptic) ODE, $\dot x = R_\lambda(x)$, is nothing but
the \emph{characteristic equation} of the  PDE, $R_\lambda[f] = u$. 
(See \cite[Chapter 1, Section 4]{fritzjohn}  for nice exposition on the role of characteristic 
curves for quasi-linear PDEs.)  As mentioned before, we cannot directly 
formulate a Fredholm framework for the application of Sard-Smale theorem to the latter
PDE. The upshot of the scheme of our proof is, however,  that we can formulate one for the
former $ODE$ by regarding the moduli space of the characteristic equation with marked points as 
a \emph{finite dimensional family} of approximate solutions of the PDE
$R_\lambda[f] = 0$ which better and better approximate the solution space of the PDE as $k \to \infty$.
We believe that our kind of practice can be applied to the study of 
other similar-type of homological equations and has independent  interest of its own.
\end{rem}
We mention that $R_\lambda[f] = u$ is always locally solvable
utilizing a Darboux coordinate system in which we have $R_\lambda = \frac{d}{dz}$.
The main question then is whether the equation is \emph{globally solvable}
on compact contact manifolds without boundary.
In fact, an equation of the type $X[f] =g$ for a vector field $X$, such as
$R_\lambda[f] = u$, is commonly called  the \emph{cohomological equation} 
 in Dynamical Systems. The cases with
locally Hamiltonian (or divergence-free) vector field $X$ have been much studied 
in relation to the \emph{area-preserving dynamics} on surfaces. (See \cite{forni97,forni02,forni21,forni22},
 for example.) See also \cite{arnold:geometrical} where the same equation is called the \emph{homological 
 equation} instead.

Our main task is then to identify the correct \emph{moduli problem}
that enables us to achieve the purpose of identifying the
residual subset $\mathfrak{C}^{\text{\rm np}}(M) \subset \mathfrak{C}(M)$ such that
$$
\ker R_\lambda = \{H \in C^\infty(M,\R) \mid dH = 0\}
$$
for all $\lambda \in \mathfrak{C}^{\text{\rm np}}(M)$. We mention that the Reeb flow preserves 
the level sets of $H \in \ker R_\lambda$. Based on this key observation, we will first investigate its contrapositive statement
under the deformation of contact forms $\lambda$,
instead of directly considering the functional equation $R_\lambda[H] = 0$. 

With all these preliminary observations and measures made,
we can relate the main problem of our study 
to the moduli problem of the Reeb trajectories, i.e., the zero set of the map
$$
\gamma \mapsto \dot \gamma - R_\lambda(\gamma).
$$
We denote by
$$
\CM_k^\R(\lambda;M)
$$
the moduli space of pointed Reeb trajectories, i.e., the set of curves
$\gamma$ satisfying the ODE $\dot x = R_\lambda(x)$ with $k$ marked points $0 < t_1 < \cdots < t_k  $.

Then motivated by the practice in the Gromov-Witten-Floer theory, we consider evaluation maps
against the level set $H^{-1}(c)$ of nonempty regular value $c$, which exists as long as 
$H$ is smooth and \emph{non-constant}. We denote $\Sigma_H(c): = H^{-1}(c)$.
For those who are familiar with the fine details \cite{floer:unregularized} of a Fredholm study of the Gromov-Witten theory,
we make the following punch line of the main scheme of our proof at this point by duplicating the table 
of comparison  between the Gromov-Witten theory
and our moduli theory of Reeb trajectories explained in Remark \ref{rem:table}.
(See Section \ref{sec:wrap-up}, especially Remark \ref{rem:table}, for the unexplained 
notations.)

\medskip

\begin{center} \Large
\begin{tabular}{|c||c|} 
\hline
\text{Gromov-Witten theory}  & \text{Moduli theory of Reeb dynamics}  \\
\hline
$u: \Sigma \to M$ &  $H: M \to \R$ \\ \hline 
$u$ somewhere injective & $H$ non-constant \\ \hline
$J \in \CJ(M)$ & $\lambda \in \mathfrak{C}(M)$  or $\mathfrak{C}(M,\xi)$\\ \hline
$\delbar_J$ & $\aleph^{(0)}_{\lambda}$ \\ \hline 
$\delbar_J u$ & $\aleph^{(0)}_{H;\lambda}$ \\ \hline
$\delbar$ & $\aleph^{(0)}$ \\ \hline
$\CM(M,J)$ & $\CM^\R(\lambda;M)$ \\ \hline
\end{tabular}
\end{center}

\medskip

We study the \emph{one-jet constraint} on $\lambda$ Hamiltonian trajectories $\gamma$
\be\label{eq:dotgamma}
\gamma(0) \in \Sigma_H(c), \quad
\dot \gamma(t_i) \in T\Sigma_H(c)\, \, i= 1, \cdots k
\ee
in the study of the moduli space $\CM^\R(\lambda;M)$. The resulting constrained
moduli space is anticipated to have its  dimension $2n+1 -k$ by the general theory of the first order ODE
$\dot x = R_\lambda(x)$ and dimension counting. It follows that each one-jet constraint at a point $t_i$ 
cuts down the dimension of the moduli space by $1$ \emph{provided the constraint is
transverse}, since then the constraint \eqref{eq:dotgamma} is of
codimension $2$. Therefore if we put the constraint \eqref{eq:dotgamma} at the points with
$k \geq 2n+2$, we anticipate that \emph{under the assumption $H$ is not identically zero},
the constrained moduli space would become empty for a generic choice of $\lambda$.

\bigskip

\noindent{\bf Acknowledgement:} We thank Y. Savelyev for the collaboration of \cite{oh-savelyev:strict-contact}.
The present work is a spin-off of our  collaboration.  We also thank Hyuk-Seok Do for the
collaboration of \cite{do-oh:reduction}  which provides  some new insight on the structure of
general contact Hamiltonian dynamics 
that plays an  important role in our proof strategy. We also thank him (and the referee as well)
for his careful reading of the previous version of the paper and pointing out the error in our translation of
the main result into the language of the foliation de Rham cohomology.  Hyun-Seok Do also provides
a counterexample presented in Subsection \ref{subsec:hyunseok}. We sincerely thank the unknown referee
for carefully reading the original manuscript and providing many helpful comments and suggestions.
Our accommodation of their comments has much simplified the exposition and improves readability of the paper.

\bigskip

\noindent{\bf Conventions and Notations:}

\medskip

\begin{itemize}
\item {(Contact Hamiltonian)} We define the contact Hamiltonian of a contact vector field $X$ to be
$- \lambda(X) =: H$. 
\item
For each given time-dependent function $H = H(t,x)$, we denote by $X_H=X_{(H;\lambda)}$ the associated contact Hamiltonian 
vector field whose associated Hamiltonian $- \lambda(X_t)$ is given by $H = H(t,x)$, and its flow by 
$\psi_H^t = \psi_{(H;\lambda)}^t$.
\item {(Reeb vector field)} We denote by $R_\lambda$ the Reeb vector field associated to $\lambda$
and its flow by $\phi_{R_\lambda}^t$. It is the same as $X_H$ with $H \equiv -1$.
\item We decompose and write $X_H = X_H^\pi - H\, R_\lambda$ with respect to the decomposition of
the tangent bundle $TM = \xi \oplus \R \{R_\lambda\}$.
\item $g_{(\psi;\lambda)}$: the conformal exponent defined by 
$\psi^*\lambda = e^{g_{(\psi;\lambda)}} \lambda$
for the contact pair $(\lambda,\psi)$.
\item $\CF_\lambda$; the Reeb foliation, i.e., the foliation of $\lambda$-Reeb trajectories.
\item $\CN_\lambda$; the leaf space of the foliation $\CF_\lambda$, i.e., the quotient space $\CF_\lambda = M/\sim$
equipped with the quotient topology.
\item {(Splitting)} We write $\alpha = V_\alpha^\pi \intprod d\lambda + h_\alpha \, \lambda$ with $h_\alpha: = \lambda(R_\lambda)$
in terms of the splitting $T^*M = (\R \{R_\lambda\})^\perp \oplus \xi^\perp$ dual to the splitting $TM = \xi \oplus \R \{R_\lambda\}$.
\item {(Variation of contact Hamiltonian vector fields)} We denote 
$$
Y_\alpha : = \delta_\lambda(R_\lambda)(\alpha) = \frac{d}{ds}\Big|_{s=0} R_{\lambda_s}, \\
$$for the germ of curves $\{\lambda_s\}_{-\epsilon < s < \epsilon}$
satisfying $ \lambda_0 = \lambda, \, \alpha = \dot \lambda_s|_{s = 0}$.
\item {(Symplectic diffeomorphisms)} We exclusively use $\phi$, never $\psi$, 
to denote a symplectic diffeomorphism whenever it appears.
\end{itemize}

\medskip

\noindent{\bf Terminologies:}
\begin{itemize}
\item We will use the two terms `derivative' and `variation' (or `first variation')
interchangeably as we feel like doing, depending on the circumstances.
\item {(Off-shell versus on-shell)} We  adopt  and freely use physicists'  standard terminologies `off-shell' 
and `on-shell' as follows:
\begin{enumerate}
\item {(Off-shell)} A treatment \emph{without assuming the `equation of motion'  being satisfied}, 
i.e., `on the general
function space', 
\item {(On-shell)} A treatment  \emph{assuming the `equation of motion'}, i.e.,
`on the moduli space of solutions'.
\end{enumerate}
\end{itemize}

\section{Basic contact Hamiltonian geometry and dynamics}
\label{sec:contact-Hamiltonian-geometry}

From the beginning, we would like to emphasize the differences of the contact Hamiltonian 
dynamics from symplectic Hamiltonian dynamics, e.g., the presence of the
\emph{vacuum dynamics} (i.e., that of constant Hamiltonian) of Reeb flow in the former case.
This has many ramifications throughout the whole contact Hamiltonian geometry and
dynamics which deviate from that of symplectic ones. In this section, 
we briefly summarize basic calculus of contact Hamiltonian geometry
partially to set up our notations and sign conventions
following \cite{oh:contacton-Legendrian-bdy}, \cite{dMV}. 

Now let   $(M,\xi)$ be a contact manifold of dimension $m = 2n+1$, which is coorientable.
We denote by $\Cont_+(M,\xi)$ the set of orientation preserving contactomorphisms
and by $\Cont_0(M,\xi)$ its identity component.  
Equip $M$ with a contact form $\lambda$ with $\ker \lambda = \xi$. 
 A choice of contact form also induces a natural measure induced by the $(2n+1)$-form
$\lambda \wedge (d\lambda)^n$.

\subsection{The conformal exponent of a contactomorphism}
\label{subsec:conformal-exponent}

For a given contact form $\lambda$, a coorientation preserving diffeomorphism $\psi$ of $(M,\xi)$
is contact if and only if it satisfies
$$
\psi^*\lambda = e^g \lambda
$$
for some smooth function $g: M \to \R$, which depends on the choice of contact form $\lambda$ of $\xi$.
Unless said otherwise, we will always assume that $\psi$ is coorientation preserving without mentioning
from now on.

\begin{defn} For given coorientation preserving contact diffeomorphism $\psi$ of $(M,\xi)$ we call
the function $g$ appearing in $\psi^*\lambda = e^g \lambda$ the \emph{conformal exponent}
for $\psi$ and denote it by $g=g_{(\psi;\lambda)}$.
\end{defn}

The following lemma is a straightforward consequence
of the identity $(\phi\psi)^*\lambda = \psi^*\phi^*\lambda$.

\begin{lem}\label{lem:coboundary} Let $\lambda$ be given and denote by $g_\psi$ the function $g$
appearing above associated to $\psi$. Then
\begin{enumerate}
\item $g_{\phi\psi} = g_\phi\circ \psi + g_\psi$ for any $\phi, \, \psi \in \Cont(M,\xi)$,
\item $g_{\psi^{-1}} = - g_\psi \circ \psi^{-1}$ for all $\psi \in \Cont(M,\xi)$.
\end{enumerate}
\end{lem}
The following  general derivative formula,  applied to \emph{autonomous Hamiltonian} in the 
present paper, plays an important role in our proof.

\begin{lem}\label{lem:dgdt} For any time-dependent Hamiltonian $H = H(t,x)$, we have
\be\label{eq:dgdt}
\frac{\del g_{(\psi_H^t;\lambda)}}{\del t} = -R_\lambda[H_t] \circ \psi_H^t.
\ee
\end{lem}
\begin{proof} Since $\lambda$ is fixed, for the simplicity of notation, we write $X_{(H_t;\lambda)} = :X_t$
and $g_{(\psi_H^t;\lambda)} = :g_t$. Then by definition, we have
$$
(\psi_H^t)^*\lambda = e^{g_t} \lambda.
$$
By differentiating this equation in $t$, we derive
$$
(\psi_H^t)^*\CL_{X_t}\lambda =  \frac{\del g_t}{\del t} e^{g_t} \lambda =  \frac{\del g_t}{\del t}
(\psi_H^t)^* \lambda
$$
which can be rewritten as
$$
\CL_{X_t} \lambda =  \frac{\del g_t}{\del t} \circ (\psi_H^t)^{-1} \lambda
$$
This in turn can be written as
$$
X_t \intprod d\lambda + d(X_t \intprod \lambda) =  \frac{\del g_t}{\del t} \circ (\psi_H^t)^{-1} \lambda.
$$
Evaluating this equation against $R_\lambda$, we have obtained
\beastar
\frac{\del g_t}{\del t} \circ (\psi_H^t)^{-1} & = & \frac{\del g_t}{\del t} \circ (\psi_H^t)^{-1} \lambda(R_\lambda)\\
& = & d(X_t \intprod \lambda)(R_\lambda) = d(-H_t)(R_\lambda) = - R_\lambda[H_t].
\eeastar
This finishes the proof.
\end{proof}

\begin{cor} We have
$$
g_{(\psi_H^T;\lambda)}(x) = \int_0^T -R_\lambda[H](\gamma(u))\, du 
$$
for the solution $\gamma(u) = \psi_H^u(x)$ with initial condition $\gamma(0) = x$.
\end{cor}

\begin{rem}
In the point of view of control theory,
we regard the contact Hamiltonian system as a \emph{controlled dynamical system} with
\begin{itemize}
\item contact form $\lambda$ as a control parameter,
\item the variation of trajectories as \emph{response} and
\item the conformal exponent $g_{(H;\lambda)}$ as a \emph{pay-off functional}.
\end{itemize}
(We refer to \cite[Chapter 1 \& 4]{evans:control} for some basic terminologies 
used in the control theory.)
\end{rem}

\subsection{Hamiltonian calculus of contact vector fields}

In this subsection, we summarize basic Hamiltonian calculus on the \emph{coorientable} contact
manifold, which as a whole will play 
an important role throughout the proofs of the main results of the present paper.
A systematic summary with the same sign convention is given in 
\cite[Section 2]{oh:contacton-Legendrian-bdy}. (See also \cite{dMV} for
a similar exposition with the same sign convention.)

\begin{defn} A vector field $X$ on $(M,\xi)$ is called \emph{contact} if $[X,\Gamma(\xi)] \subset \Gamma(\xi)$
where $\Gamma(\xi)$ is the space of sections of the vector bundle $\xi \to M$.
We denote by $\mathfrak X(M,\xi)$ the set of contact vector fields.
\end{defn}

We have the following unique decompositions of
a  vector field $X$ and one-forms $\alpha$ in terms of \eqref{eq:TM-decomposition} and
\eqref{eq:T*M-decomposition} respectively.  They will play an important role in 
 various calculations entering in the calculus of contact
 Hamiltonian geometry and dynamics.
 
 \begin{lem}\label{lem:decomposition}
  Let $X$ be a vector field and $\alpha$ a one-form on $(M,\lambda)$. Then
 we have the unique decompositions 
\be\label{eq:X-decompose}
X = X^\pi + \lambda(X) R_\lambda,
\ee
and 
\be\label{eq:alpha-oneform}
\alpha = \alpha^\pi + \alpha(R_\lambda) \lambda
\ee
respectively. Furthermore we have  the unique representation
$$
\alpha^\pi = V_\alpha^\pi \intprod d\lambda
$$
for a section $V_\alpha^\pi \in \Gamma(\xi)$ uniquely determined by $\alpha^\pi$ and vice versa.
\end{lem}

Next, the condition $[X,\Gamma(\xi)] \subset \Gamma(\xi)$ is
 equivalent to the condition that
there exists a smooth function $g: M \to \R$ such that
\be\label{eq:CLX=glambda}
\CL_X \lambda = g \lambda.
\ee
\begin{defn} Let $\lambda$ be a contact form of $(M,\xi)$, and $X$ a contact vector field.
The associated function $H$ defined by
\be\label{eq:contact-Hamiltonian}
H = - \lambda(X)
\ee
is called the \emph{$\lambda$-contact Hamiltonian} of $X$. We also call $X$ the
\emph{$\lambda$-contact Hamiltonian vector field} or simply a contact Hamiltonian 
vector field associated to $H$.
\end{defn}
We alert readers that under our sign convention, the $\lambda$-Hamiltonian $H$ of the Reeb vector field $R_\lambda$
as a contact vector field becomes the constant function $H = -1$.

For a given contact form $\lambda$,
any smooth function $H$ gives rise to a contact vector field $X$ determined by \eqref{eq:XlambdaH},  
which we denote by
$$
X_H= X_{(H;\lambda)}
$$
to emphasize the $\lambda$-dependence of the expression $X_H$. We highlight
the fact that unlike the symplectic case, this correspondence is one-one with
no ambiguity of addition by constant. 
We will need to understand how the Hamiltonian vector field $X_{(H;\lambda)}$ varies
under the change of $(\lambda,H)$  for our further discussion.

We next apply the aforementioned decomposition of one-forms 
\eqref{eq:alpha-oneform} to exact one-forms.

\begin{prop} Equip $(M,\xi)$ with a contact form $\lambda$.
Let $X$ be a contact vector field and consider the decomposition
$X = X^\pi + \lambda(X) R_\lambda$.  If we set $H = -\lambda(X)$, then $dH$ 
satisfies the equation 
\be\label{eq:dH}
dH = X^\pi  \intprod d\lambda  + R_\lambda[H]\, \lambda.
\ee
\end{prop}

\begin{lem}\label{cor:LXtlambda} Let $X$ be a contact vector field with $\CL_X \lambda = g \lambda$,
and let $H$ be the associated contact Hamiltonian. Then $g = -R_\lambda[H]$.
\end{lem}
The following ampleness of the contact Hamiltonian vector fields will be used 
frequently throughout the paper.

\begin{lem}[Ampleness of Hamiltonian vector fields]\label{lem:Ham-ampleness} We have
$$
\{X_H(x) \in T_x M \mid H \in C^\infty(M,\R) \} = T_x M
$$
at all $x \in M$.
\end{lem}
\begin{proof} Recall the two formulae
\beastar
X_H & = & X_H^\pi -H R_\lambda, \\
dH & = &  X_H^\pi \intprod d\lambda + R_\lambda[H] \lambda.
\eeastar
Recall that the set of vectors
$X_H^\pi(x) \intprod d\lambda$ spans $(\R \{R_\lambda \})^\perp$ with $H \in C^\infty(M,\R)$, and the
map
$$
X_H^\pi(x) \mapsto X_H^\pi(x) \intprod d\lambda
$$
is an isomorphism between $\xi_x \subset T_xM$ and $(\R \{R_\lambda(x)\})^\perp \subset T_x^*M$. 
Since at any point $x \in M$ we can find
a function $H$ with $H(x) = 0$ and $dH(x)$ can be arbitrary element in $T^*_xM$, we have
$$
\{X_H(x) \in T_x M \mid H \in C^\infty(M,\R) \} \supset \xi_x.
$$
On the other hand, by considering $H$ with $dH(x) = 0$ but $H(x) \neq 0$, we have
$$
X_H(x) = -H(x) R_\lambda(x)
$$
which shows that the subset also contains the span $\R \{R_\lambda(x)\}$. This finishes the proof.
\end{proof}

\subsection{Contact triad, triad metric and  connection}
\label{subsec:triad-connection}

We will need to study the deformation of the cokernel of the operator under the change of contact form
 $\lambda$. For this purpose, we use the triad metric and the $L^2$-adjoint map of the linear operator $R_\lambda$. We also  equip $M$ with the contact Liouville measure 
$\mu_\lambda$ associated to the contact form $\lambda$.

A brief explanation on the standard definition of the triad metric is now in order.
In the presence of the contact form $\lambda$,
one considers the set of endomorphisms $J: \xi \to \xi$ that are compatible with $d\lambda$ in
the sense that the bilinear form $g_\xi = d\lambda(\cdot, J\cdot)$ defines
a Hermitian vector bundle $(\xi,J, g_\xi)$ on $M$.
We call such an endomorphism $J$ a $CR$-almost complex structure.
We extend $J$ to an endomorphism of $TM$ by setting $JX_\lambda=0$.

\begin{defn}[Contact triad metric]\label{defn:triad-metric}
We call the triple $(M, \lambda, J)$ a \emph{contact triad} and equip $M$ with the Riemannian metric
$$
g_\lambda = d\lambda(\cdot, J \cdot) +  \lambda\otimes\lambda
$$
which we refer to as the \emph{contact triad metric}.
\end{defn}

The following coincidence of the Liouville volume form with
the metric volume form associated to the contact triad metric illustrates
 how the contact calculus and the Riemannian calculus interact \emph{under the
 usage of contact triad metric}.

\begin{lem}[Lemma 3.1 \cite{oh-savelyev:strict-contact}]
\label{lem:dvol=dmulambda} Consider the contact triad $(M,\lambda,J)$ and
its associated triad metric $g$ and the volume form $d\vol_g$.  Then we have
$$
d \mu_\lambda = d\vol_g.
$$
\end{lem}
\begin{proof}
To prove the coincidence of the two volume forms, we consider the 
Darboux frame
$$
\{E_1, \cdots, E_n, F_1, \cdots, F_n, R_\lambda\}
$$
such that $E_i, \, F_j \in \xi$ and $F_j = J E_j$ and so $d\lambda(E_i,F_j) = \delta_{ij}$,
and compare the two values 
\beastar
d\mu_\lambda(E_1, \cdots, E_n, F_1, \cdots, F_n, R_\lambda), \\
d\vol_\lambda (E_1, \cdots, E_n, F_1, \cdots, F_n, R_\lambda).
\eeastar
 Obviously we have
$$
d\mu_\lambda(E_1, \cdots, E_n, F_1, \cdots, F_n, R_\lambda) =  \frac1{n!}\left(n!\,  \lambda(R_\lambda)
\cdots d\lambda(E_1,F_1)\cdots d\lambda(E_n,F_n) \right)= 1.
$$
On the other hand, the above frame is also an orthonormal frame of
the triad metric $g$ and hence $d\vol(E_1,\cdots, E_n, F_1, \cdots, F_n, R_\lambda) = 1$.
This finishes the proof.
\end{proof}

We will also need to take covariant derivatives of various sections along the Reeb trajectories. 
For the purpose of the present paper, we may use either associated Levi-Civita
connection or the \emph{contact triad connection} introduced in 
\cite{oh-wang:connection}.  One important property both connections share is the following
killing property of the Reeb vector field, the proof of which we refer readers to
\cite{blair} and to \cite{oh-wang:connection} respectively.

\begin{lem}\label{lem:Rlambda-killing} Let $\nabla$ be either the Levi-Civita connection or the
contact triad connection associated to the triad metric of a contact triad $(M,\lambda,J)$.
Then 
$$
\nabla_{R_\lambda}R_\lambda = 0.
$$
\end{lem}

\section{Cohomological formulation of the functional equation $R_\lambda[f] = u$}
\label{sec:foliation-deRham}

Next, we study the question of solvability of  the following first order  PDE \eqref{eq:Rlambdaf=u}
$$
R_\lambda[f] = u,
$$
which is a fundamental global question on its own concerning the Reeb dynamics.

Let $\mu_\lambda$ be the contact Liouville measure of $\lambda$.
\begin{lem}\label{lem:intRlambdaf=0} Assume $M$ is closed.
For any contact pair $(\lambda,f)$, we have
\be\label{eq:intRlambdaf=0}
\int_M R_\lambda[f]\, d \mu_\lambda = 0.
\ee
\end{lem}
\begin{proof} Utilizing the Reeb invariance of the measure $\mu_\lambda$, $\CL_{R_\lambda} \mu_\lambda = 0$, we rewrite
\beastar
\int_M R_\lambda[f]\, d \mu_\lambda & = &  \int_M \CL_{R_\lambda}(f)\, d \mu_\lambda
=  \int_M \CL_{R_\lambda}(f\, d \mu_\lambda)  \\
& =  & \int_M d\left(R_\lambda \intprod (f \mu_\lambda) \right) = 0
\eeastar
by Stokes's lemma for the vanishing.
This finishes the proof.
\end{proof}

Lemma \ref{lem:intRlambdaf=0} shows that 
$$
\coker R_\lambda \supset \{\text{\rm constant functions}\}.
$$
On the other hand, Theorem \ref{thm:nonprojectable-intro} can be rephrased as
$$
\ker R_\lambda = \{ \text{\rm constant functions}\} \cong \R.
$$

\subsection{Foliation de Rham cohomology of Reeb foliations}

By definition, any contact manifold equipped with a contact form $\lambda$ becomes
an \emph{exact presymplectic manifold} $(M,d\lambda)$ with nullity 1, i.e., the dimension of
$\ker d\lambda$ is one. The associated null foliation is precisely the Reeb foliation generated by
the Reeb trajectories. (See \cite{molino88}, \cite{tondeur88,tondeur97} for the basic definitions 
on the leafwise de Rham cohomology in general.)
This leafwise de Rham cohomology $H^\bullet(\CF_\lambda)$ is determined
by the presymplectic structure $d\lambda$  on $M$ manifold $(M,d\lambda)$,
or more precisely it depends only on the foliation $\CF_\lambda$.
It is called the {\it foliation de Rham cohomology of (the null foliation of) $(M,d\lambda)$}
in \cite{oh-park:coisotropic}. 

In our current context of contact manifold equipped with a contact form
$\lambda$, it give rise to the two-term complex
$$
0 \to \Omega^0(\CF_\lambda) \to \Omega^1(\CF_\lambda) \to 0.
$$
associated to the Reeb foliation of $\lambda$. Therefore the only possibly non-zero cohomology groups are
$$
H^0(\CF_\lambda), \quad H^1(\CF_\lambda).
$$
We recall
\bea
\Omega^0(\CF_\lambda) & = &  C^\infty(M), \nonumber \\
 \Omega^1(\CF_\lambda) & = & \{f \, \lambda \mid f \in C^\infty(M)\}.
\eea

The following explicit formula for the differential $d_\CF$ is important for our study of 
the solvability question of the equation $R_\lambda[f] = u$.
\begin{prop}  Let
 $d_\CF: \Omega^*(\CF_\lambda) \to \Omega^*(\CF_\lambda)$ be 
the leafwise differential. Then we have
\beastar
Z^0(\CF_\lambda) & = & \{ f \, \lambda \mid R_\lambda[f] = 0\}, \\
 B^1(\CF_\lambda) & = & \{u \, \lambda \mid u = R_\lambda[f], 
\, \text{\rm for some } \, f \in C^\infty\} \\
Z^1(\CF_\lambda) & = & \{ u \, \lambda \mid u \in C^\infty(M) \}.
\eeastar
In particular, we have $H^0(\CF_\lambda) = Z^0(\CF_\lambda)$ and
\be\label{eq:H1-FF}
H^1(\CF_\lambda) = \frac{\{ f \, \lambda \mid R_\lambda[f] = 0\}} {\{R_\lambda[f] \, \lambda \mid f \in C^\infty(M)\}}.
\ee
\end{prop}
\begin{proof} Let $f \in \Omega^0(\CF_\lambda) = C^\infty(M)$. By definition, we have
$$
\Omega^1(\CF_\lambda) \cong C^\infty(M) \cdot \lambda.
$$
 Therefore we compute
$$
d_\CF(f) = d_\CF(f)(R_\lambda)\, \lambda = (df|_{T\CF_\lambda})(R_\lambda)\, \lambda = df(R_\lambda)\, \lambda= R_\lambda[f]\, \lambda.
$$
On the other hand, since $\rank T\CF_\lambda = 1$, we have
$$
d_\CF(\Omega^1(\CF_\lambda))  = \{0\}.
$$
This implies
\beastar
H^0(\CF_\lambda) & = & \ker d_\CF = \{f \in C^\infty(M, \R) \mid R_\lambda[f] =  0\}\\
H^1(\CF_\lambda) & = & \frac{\Omega^1(\CF_\lambda)}{\{R_\lambda[u] \, \lambda \mid u \in C^\infty(M) \}}.
\eeastar
This finishes the proof.
\end{proof}
 
We now state some immediate consequences of Theorem \ref{thm:nonprojectable-intro}.

\begin{cor}\label{cor:non-projectable} Assume $M$ is connected.
Let $\lambda$ be a contact form arising in Theorem \ref{thm:nonprojectable-intro}. Then
 the operator 
$$
R_\lambda:C^\infty(M,\R) \to C^\infty(M,\R).
$$
has its kernel given by $\ker R_\lambda =\{ \text{\rm constant functions on $M$}\} \cong \R$.
In particular,
$$
H^0(\CF_\lambda) \cong \R.
$$
and
$$
\mathfrak{cont}^{\text{\rm st}} (M,\lambda) = \R\{R_\lambda\} \cong \R.
$$
 \end{cor}
 \begin{proof} These statements are obviously, by now, derived  from Theorem \ref{thm:nonprojectable-intro}.
 \end{proof}

\subsection{Cokernel of the cohomological equation $R_\lambda[f] = u$}
\label{subsec:hyunseok}

On the other hand, computation $H^1(\CF_\lambda)$, i.e., the solvability of 
the equation $R_\lambda[f] = u$ or the study of the cokernel of $R_\lambda$ 
crucially depends on the given function $u$. 

In this subsection, we show that in general the $H^1(\CF_\lambda)$ becomes infinite
dimensional whenever the contact manifold $(M,\lambda)$ carries a closed Reeb orbit.
This proof is explained to the author by Hyun-Seok Do in a private conversation \cite{do-example}.

\begin{thm}\label{thm:do} Let $(M, \xi)$ be a compact contact manifold and $\lambda$ be its contact form.
Suppose that $\lambda$ admits a closed Reeb orbit. Then we have
$H^1(\CF_\lambda)$ is infinite dimensional.
\end{thm}
\begin{proof} Recall $H^1(\CF_\lambda) \cong \coker R_\lambda$.
We will show that there is an infinite dimensional subspace of 
$C^\infty(M;\R)$ consisting of $u$ for which the equation $R_\lambda[f] = u$
does not carry a solution. 

Let $\gamma: [0,T] \to M$ be a simple closed Reeb orbit
of period $T > 0$. Denote by $C_\gamma \subset M$ its image which is a one-dimensional smooth compact
submanifold, and let $V \subset \overline V \subset U$ be a
pair of  relatively compact tubular neighborhoods thereof. Assume $M \setminus U$ is
a nonempty open subset, and take any cut-off function $\rho: M \to \R$
satisfying 
$$
\supp \rho \subset  U, \quad \rho \equiv 1 \, \quad{\rm  on }\,\, \overline V.
$$ 
The following will obviously complete the proof of the theorem.

\begin{lem} $\rho - a$ cannot be contained in
the image of $R_\lambda$ for any constant $a$. 
\end{lem}
\begin{proof} Obviously since we assume $\overline U \neq M$, $\rho$ cannot be a constant function.
Furthermore since $\rho(x) = 1$ on $\overline V \supset C_\gamma$, we also have
$$
\int_0^T \rho(\gamma(t))\, dt =\int_0^T 1\, dt = T > 0.
$$
Now suppose to the contrary that $\rho - a$ is in the image of $R_\lambda$. Then we have
$$
\rho -a = R[f]
$$
for some $f \in C^\infty(M;\R)$. Then we compute
\beastar
T & = & \int_0^T \rho(\gamma(t))\, dt = \int_0^T (R_\lambda[f](\gamma(t)) + a) \, dt\\
& = & \int_0^T \frac{d}{dt}(f\circ \gamma)(t)\, dt + a T = (f(\gamma(T) - f(\gamma(0))) + aT = aT
\eeastar
where the last equality follows from the assumption that $\gamma(T) = \gamma(0)$.
Since $T \neq 0$, this implies we must have $a = 1$. This then implies
$$
\rho = R_\lambda[f] + 1.
$$
Then integration of this equation over $M$ gives rise to
$$
\int_M \rho \mu_\lambda = \int_M R_\lambda[f]\, \mu_\lambda + \vol(\lambda)
$$
where $\vol(\lambda) = \int_M \mu_\lambda$.
Since $\int_M R_\lambda[f]\, \mu_\lambda = 0$, we obtain
$$
\int_M \rho \,\mu_\lambda = \vol(\lambda)
$$
which is absurd since $\supp \rho \subset U$, $0 \leq \rho \leq 1$, and $M \setminus \overline U$ is open
so that the open set has nonzero Liouville measure. This finishes the proof.
\end{proof}

By now we have shown $\V_\gamma \cap \Image R_\lambda = \emptyset$ for the subset $\V_\gamma \subset C^\infty(M;\R)$ defined by
$$
\V_\gamma: = \{\rho - a \mid a \in \R,\, M \setminus \supp \rho \, \text{\ \rm is  nonempty open}, \, 
 C_\gamma \subset \rho^{-1}(1) \}.
$$
It is easy to see that $\V_\gamma$ forms an infinite dimensional subspace of $C^\infty(M;\R)$
by varying the support of $\rho$. This finishes the proof.
\end{proof}

\begin{rem} Recall that the standing hypothesis of Theorem \ref{thm:do} is precisely 
the celebrated Weinstein's conjecture in contact topology \cite{alan:conjecture}.
Therefore if the conjecture holds, $H^1(\CF_\lambda)$ is always infinite dimensional for 
all contact forms by this theorem. Equivalently, by taking its contrapositive, if $(M,\lambda)$
has finite dimensional $H^1(\CF_\lambda)$, $\lambda$ does not carry any closed Reeb orbit.
\end{rem}

The following is an immediate corollary of Theorem \ref{thm:do} and this remark,
when combined with Taubes' proof \cite{taubes} of Weinstein conjecture on 
3-dimensional contact manifolds.

\begin{cor} Let $(M,\xi)$ be any compact contact 3-manifold. Then 
$$
\dim H^1(\CF_\lambda,M) = \infty
$$
for all contact forms $\lambda$ of $\xi$. More generally, the same also holds for all higher dimensions, 
\emph{whenever Weinstein's conjecture holds.}
\end{cor}

\section{Big and small kinetic theory phase spaces}
\label{sec:phasespace}

In \cite{do-oh:reduction}, the terminology of \emph{contact kinetic theory phase space}, big and small,
is introduced. Adopting the same terminology restricted to the corresponding configuration space, we
recall some basic definitions and properties therefrom that we are going to use.

Let $M$  be a compact connected smooth contactable manifold of dimension $2n+1$.
We denote by $\mathfrak{C}(M)$ the set of
maximally nondegenerate one-forms.  It follows that $\mathfrak{C}(M)$ is an open subset 
of $\Omega^1(M)$ and so its tangent space 
is canonically identified with $\Omega^1(M)$ itself.  We denote this canonical 
inclusion by 
\be\label{eq:iota}
\iota:\mathfrak{C}(M) \hookrightarrow \Omega^1(M).
\ee
When $M$ is equipped with an orientation, we also consider the subset
$\mathfrak C^+(M)$ consisting of $\lambda$ with
$$
\vol(\lambda): = \int_M  \lambda \wedge (d\lambda)^n > 0.
$$
Note that we have \emph{canonical} identification of its tangent space given by
\be\label{eq:TlambdaCM}
T_\lambda \mathfrak{C}^+(M) \cong \Omega^1(M)
\ee
induced by the assignment of $\alpha$ to the corresponding germ of curves of contact forms
$$
\{t\mapsto \lambda_t\}_{-\epsilon< t < \epsilon}, \quad \lambda_0 = \lambda, \, \dot \lambda_0 = \alpha
$$
noting that nondegeneracy of
one-form is an open condition in $C^\infty$ topology on compact $M$.

\begin{defn} [Big phase space]\label{defn:big} We call the set of contact pairs  $(\lambda,\psi)$
the $\textit{big  phase space}$ of contact pairs, and denote it by
$$
\mathfrak{Cont}(M) : = \bigcup_{\lambda \in \mathfrak{C}(M)} \{\lambda\} \times \Cont(M,\lambda)
\subset \mathfrak{C}(M) \times \Diff(M).
$$
We also consider its subset
$$
\mathfrak{Cont} (M)_0 : = \bigcup_{\lambda \in \mathfrak{C}(M)} \{\lambda\} \times \Cont(M,\lambda)_0
\subset \mathfrak{C}(M) \times \Diff(M)_0
$$
where $\Cont(M,\lambda)_0$ (resp. $\Diff(M)_0$) is the identity component of $\Cont(M,\lambda)$
(resp. $\Diff(M)$).
\end{defn}

By restricting to the set of contact pairs $(\lambda,\psi)$ with $\ker \lambda = \xi$
for a fixed contact structure $\xi$ on $M$,  we define
the subset $\mathfrak{C}(M,\xi) \subset \mathfrak{C}(M)$ to be the one
consisting such contact forms. We then define the following small phase space.

\begin{defn}[Small phase space]\label{defn:small}
Let $(M,\xi)$ be a coorientable contact manifold.
We call the set of contact pairs $(\lambda,\psi)$ associated to the given contact structure $\xi$
the \emph{small phase space} associated to the contact structure $\xi$.   
We denote its subset consisting of $(\lambda,\psi)$ with 
$\lambda \in \mathfrak{C}(M,\xi)$ for a given contact structure $\xi$ by
\be\label{eq:small-contact-pair-xi}
\mathfrak{Cont}(M,\xi) 
= \bigcup_{\lambda \in \mathfrak{C}(M,\xi)} \{\lambda\} \times \Cont(M,\xi).
\ee
\end{defn}

The following is a folklore but we cannot locate its proof in the literature. For
readers' convenience and for completeness' sake, we give its proof in Appendix.

\begin{lem}\label{lem:Tlambda} The set $\mathfrak{C}(M)$ carries
the Frechet submanifold structure modeled by
$$
(\R^{2n})^* \times \text{\rm Symp}^2(\R^{2n})
$$
where $\text{\rm Symp}^2(\R^{2n})$ is the set of symplectic bilinear forms.
\end{lem}

Since the following discussion is easier for the big phase space case, we will focus on the 
case of small phase spaces and just leave the corresponding 
obvious statements of the big phase space to the readers.

We consider the canonical inclusion map
\be\label{eq:iotasm}
\iota_\xi: = \iota|_\xi : \mathfrak{C}(M,\xi) \hookrightarrow \Omega^1(M).
\ee
Note that the domain depends on $\xi$ but the map is just the restriction of the canonical inclusion
\eqref{eq:iota}. We also recall readers that
the subset $\mathfrak C(M,\xi) \subset \Omega^1(M)$ is \emph{neither a linear subspace nor
an open subset of $\Omega^1(M)$}, but a principal homogeneous space of the group $C^\infty(M,\R_+)$.
In particular we have
$$
\mathfrak C(M,\xi) \cong C^\infty(M,\R_+),
$$
\emph{but not canonically}. On the other hand 
in the tangent space level, we have a canonical 
$$
T\mathfrak C(M,\xi) \to \mathfrak{C}(M,\xi) \times C^\infty(M,\R) ; \quad (\lambda,h\lambda) \mapsto 
(\lambda, h)
$$
that gives rise to the commutative diagram
$$
\xymatrix{
T\mathfrak C(M,\xi) \ar[r]\ar[d]  &\mathfrak C(M,\xi) \times C^\infty(M,\R) \ar[d] &\\
\mathfrak C(M,\xi) \ar[r]^{=} & \mathfrak C(M,\xi)
}
$$
which is equivariant under the canonical action of $C^\infty(M,\R_+)$ on $T\mathfrak C(M,\xi)$
induced from that of $C^\infty(M,\R_+)$ and the diagonal action thereof on 
$ \mathfrak{C}(M,\xi) \times C^\infty(M,\R)$.

We denote by $\mathfrak{C}^+(M,\xi)$ is the set of positive contact forms of
the cooriented contact manifold $(M,\xi)$. Then we have
$$
\mathfrak{C}(M,\xi) = \mathfrak{C}^+(M,\xi) \sqcup \mathfrak{C}^-(M,\xi); \quad 
\mathfrak{C}^-(M,\xi): = - \mathfrak{C}^+(M,\xi).
$$

\begin{prop}\label{prop:Tlambda-xi}
Assume $M$ is orientable and a coorientable contact structure $\xi$ is given.
Then we have the following:
\begin{enumerate}
\item $\mathfrak{C}(M,\xi)$ 
is a smooth submanifold of $\mathfrak{C}(M)$ and so carries a
natural Frechet manifold structure induced by the inclusion map \eqref{eq:iotasm}. 
\item 
The tangent space of $\mathfrak C(M,\xi)$ at $\lambda$ is the subspace of $\Omega^1(M)$ given by
\be\label{eq:Tlambda-CMxi}
T_\lambda \mathfrak{C}(M,\xi)
= \{\alpha \in \Omega^1(M) \mid \alpha = h \lambda, \, h \in C^\infty(M,\R)\}.
\ee
\end{enumerate}
\end{prop}
\begin{proof} 
Now we provide the description \eqref{eq:Tlambda-CMxi} of the tangent space in the rest of the proof.
We  note that the set $C^\infty(M,\R_+)$ is not a linear space but a multiplicative 
unital group, and have the natural logarithm map
$$
\log: C^\infty(M,\R_+) \to C^\infty(M,\R); \quad f \mapsto \log f
$$
to its tangent space at $1$ which can be canonically identified with $C^\infty(M,\R)$ via the 
correspondence 
 $$
 g \mapsto \{e^{t g}\}_{-\epsilon < t < \epsilon} 
 $$
Then it follows that   $ \mathfrak{C}^+(M,\xi)$ is a principal homogeneous space of the 
Lie group $C^\infty(M,\R_+)$. In particular any element $\lambda$ thereof
can be uniquely written as $\lambda = f\, \lambda_0$ for some $f > 0$ if we fix a reference
contact form $\lambda_0$.  This in particular gives rise to smooth one-to-one correspondence  
\be\label{eq:small-Upsilon-lambda0}
\CI_{\lambda_0} : C^\infty(M,\R_+) \to \mathfrak C^+(M,\xi) 
\ee
given by 
$$
\CI_{\lambda_0}(f) : =  f \lambda_0
$$
(which obviously depends on the choice of the reference form $\lambda_0$).  
The map in particular identifies the Frechet manifold $\mathfrak C^+(M,\xi)$ 
with $C^\infty(M,\R_+)$.

Therefore any germ of curves
$\{\lambda_t\}_{-\epsilon < t < \epsilon}$ at $\lambda$ can be written as 
$$
t \mapsto e^{g_t}\lambda,
$$
and a first variation or a tangent vector at $\lambda$ as
$$
\delta \lambda = \frac{d e^{g_t}}{d t}\Big|_{t = 0} \cdot \lambda = \frac{d g_t}{d t}\Big|_{t = 0} \cdot \lambda
$$
By setting $h: = \frac{dg_t}{d t}\Big|_{t = 0}$, we have shown \eqref{eq:Tlambda-CMxi}
as a set. We summarize the above discussion into the following commutative diagram
\be\label{eq:diagram1}
\xymatrix{
C^\infty(M,\R_+) \ar[r]^{f \mapsto \log f} \ar[dr]^{f \mapsto f\lambda}  
 &C^\infty(M,\R) \ar[d]^{g \mapsto e^g \lambda} \ar[r]^{g \mapsto g\, \lambda} &T_\lambda \Omega^1(M) \ar[d]^{g\lambda
 \mapsto e^g \lambda}\\ 
& \mathfrak{C}^+(M,\xi)  \ar[r]_{i}  & \Omega^1(M) &
}
\ee
where the map $i$ is the canonical inclusion map.
This provides the description of the tangent space of $\mathfrak{C}(M,\xi)$.

Statement (1) for $\mathfrak{C}(M,\xi)$ then directly follows from the implicit function theorem \cite{sergeraert}
by observing that the assignment 
$$
(\lambda,\psi) \mapsto (\ker \lambda,\ker \psi^*\lambda)
$$
as a map $\mathfrak{Diff}(M)$ to $\Gamma(\Gr_{2n}(M)) \times \Gamma(\Gr_{2n}(M))$ is 
transverse to the diagonal 
$$
\Delta_{\Gamma(\Gr_{2n}(M))} \subset \Gamma(\Gr_{2n}(M)) \times \Gamma(\Gr_{2n}(M))
$$
the proof of which easily follows from the ampleness, Lemma \ref{lem:Ham-ampleness}.
\end{proof}
 
The following is an immediate corollary of Proposition \ref{prop:Tlambda-xi}.
\begin{cor}\label{cor:small-smooth}
Both $\mathfrak{Cont}(M)$ and $\mathfrak{Cont}(M,\xi)$ are
smooth submanifolds of the product
$$
\mathfrak{Diff}(M) = \Omega^1(M) \times \Diff(M)
$$
in the Frechet sense. Furthermore  their tangent spaces at $(\lambda, \psi)$ are given by
$$
T_{(\lambda,\psi)} \mathfrak{Cont} (M)  \cong \Omega^1(M) 
\times \mathfrak{X}(M,\xi_\lambda)
$$
and 
$$
T_{(\lambda,\psi)} \mathfrak{Cont} (M,\xi) \cong 
 C^\infty(M,\R)\cdot \{\lambda\} \times \mathfrak{X}(M,\xi) 
 $$
respectively.
\end{cor}
\begin{proof} Since the case of $\mathfrak{Cont}(M)$ is easier, we focus on the case 
$\mathfrak{Cont}(M,\xi)$. Recall the definition
$$
\mathfrak{Cont}(M,\xi) = \bigcup_{\lambda \in \mathfrak{C}(M,\xi)} \{\lambda\} \times \Cont(M,\xi)
$$
which is a smooth fiber bundle over $\mathfrak{C}(M,\xi)$.
Proposition \ref{prop:Tlambda-xi} equips $\mathfrak{C}(M,\xi)$ with a natural smooth structure, and the fiber
of the bundle is canonically the same as $\Cont(M,\xi)$ independent of the choice of 
$\lambda \in \mathfrak{C}(M,\xi)$.
The corollary for $\mathfrak{Cont}(M,\xi)$ immediately follows from this observation.
\end{proof}

\section{Variations of $\lambda$-Reeb vector field}

For our purpose, it is crucial to analyze the way how the Reeb vector field $R_\lambda$,
varies under the variation of contact form $\lambda$. This consideration is a new important aspect of
contact Hamiltonian dynamics that is not shared by the symplectic Hamiltonian dynamics.
We believe that this variational formula applied to general contact Hamiltonian 
vector fields $X_{(H;\lambda)}$ will be useful in the future study of contact Hamiltonian
dynamics. 

Such a formula  has been implicitly studied in relation to the deformation of moduli spaces of 
pseudoholomorphic curves on the symplectization before (see \cite[Section 7]{meilin-yau} for example), 
but we could not locate a reference containing the formula such as \eqref{eq:Yalpha=} below. 

We consider and denote the variation of $\lambda$ by
$$
\alpha = \delta \lambda: = \frac{d}{dt}\Big|_{t = 0} \lambda_t
$$
for the germ of paths $\{\lambda_t\}_{-\epsilon t < \epsilon}$ at $\lambda_0 = \lambda$.
By definition,  the variation of $R_\lambda$ associated to the germ of curves $s \mapsto \lambda_s$ 
 at $\lambda|_{s=0} = \lambda_0$ with $\dot \lambda|_{s = 0} = \alpha$ is  given by
 \be\label{eq:Yalpha-defn}
 Y_\alpha:  = (\delta_\lambda R_\lambda)(\alpha) : = \frac{d}{dt}\Big|_{t=0} R_{\lambda_t}.
 \ee
 Then $Y_\alpha$  is uniquely determined by the equation
\be\label{eq:Yalpha}
\begin{cases}
Y_\alpha \intprod \lambda = - R_\lambda \intprod \alpha\\
Y_\alpha \intprod d\lambda = - R_\lambda \intprod d\alpha.
\end{cases}
\ee

The following is a very suggestive formula
for the first variation of the Reeb vector field $R_\lambda$ \emph{inside the small phase space.}

\begin{prop}\label{prop:Yalpha} Let $\lambda \in \mathfrak{C}(M)$.
Suppose $\alpha = h \, \lambda$. Then 
\be\label{eq:Yalpha}
Y_\alpha = X_{(h;\lambda)}
\ee
\end{prop}
\begin{proof} 
Substituting the expression $\alpha = h \lambda$ into \eqref{eq:Yalpha-defn},
we obtain
\be\label{eq:Y-lambda=}
Y_\alpha \intprod \lambda = - R_\lambda \intprod \alpha = -h
\ee
On the other hand, we compute
$$
d\alpha = h d\lambda + dh \wedge \lambda.
$$
Therefore, from the second equation of \eqref{eq:Yalpha-defn}, we obtain
$$
Y_\alpha  \intprod  d\lambda = -R_\lambda \intprod d\alpha
=   - R_\lambda \intprod (h d\lambda + dh \wedge \lambda)
= dh - R_\lambda[h] \, \lambda.
$$
By adding the two, we have derived
$$
\begin{cases}
Y_\alpha \intprod \lambda = - R_\lambda \intprod \alpha = -h,\\
Y_\alpha  \intprod  d\lambda = dh - R_\lambda[h] \, \lambda.
\end{cases}
$$
By the defining equation \eqref{eq:deltaXHlambda} applied to $H = h$,
the proof is finished.
\end{proof}

The following is an immediate corollary which is the small space analog to Proposition \ref{prop:Yalpha}.
We mention that if $\lambda \in \mathfrak{C}(M,\xi)$, the variation $\alpha = h \, \lambda$
is tangent to $\mathfrak{C}(M,\xi)$.

\begin{cor}\label{cor:Yalpha-small} Let $\lambda \in \mathfrak{C}(M,\xi)$ and
let $\alpha =  h\, \lambda$. Then we have
\be\label{eq:Yalpha=}
Y_\alpha: = \delta_\lambda(R_\lambda)(\alpha) =  X_{(h;\lambda)}^\pi - h R_\lambda = X_{(h;\lambda)}.
\ee
\end{cor}

\begin{rem}
More generally, we can derive the formula for the first variation of the general contact Hamiltonian vector field $X_{(H;\lambda)}$ by varying 
the defining equation of which we recall from \eqref{eq:XlambdaH} is
\be\label{eq:deltaXHlambda}
\begin{cases}
X \intprod \lambda = -H \\
X \intprod d\lambda = dH - R_{\lambda}[H] \lambda.
\end{cases}
\ee
Since we will not need this general variations, we do not discuss it in the present paper.
\end{rem}

\part{On the big phase space}
\label{part:big}

In this part, we will consider the contact pair $(\lambda,\psi)$
 allowing the contact distribution $\xi := \ker \lambda$ to vary. We call the set of such
pairs the \emph{big phase space} of contact pairs. (See \cite{do-oh:reduction} for the 
advent of this usage of the term big and small phase spaces.)

\section{Moduli space of pointed Reeb trajectories}
\label{sec:moduli}

As mentioned in the introduction, solvability of the cohomological equation 
$$
R_\lambda[f] = u
$$
is closely tied to the existence of strict contact flow of the contact form $\lambda$
which would violate the transversality of the intersection of Reeb trajectories against the 
discriminant in the jet level of any finite order. 
This necessitates us to study the space of solutions of Reeb trajectories as a whole. We will
regard its study as a moduli problem the formulation of which is now in order.

A (complete) Reeb trajectory issued at $x_0$  is a map $\gamma: \R \to M$ that satisfies the  ODE
\be\label{eq:Reeb-ODE}
\begin{cases}
\dot \gamma(t) = R_\lambda (\gamma(t))\\
\gamma(0) = x_0.
\end{cases}
\ee
We regard the problem of solving this ODE as a moduli problem. 
It follows that  any solution $\gamma$ of \eqref{eq:Reeb-ODE} satisfies the derivative bound
\be\label{eq:derivative-bounds}
\left|\nabla_t\dot \gamma \right| \leq  \|D R_\lambda\|_{C^0(M)} |\dot \gamma(t)|
\ee
for all $t \in \R$ where $\|D R_\lambda\|=\|D R_\lambda\|_{C^0(M)}$ is the operator norm of the
covariant derivative $\nabla R_\lambda: TM \to TM$ 
$$
 \|D R_\lambda\| = \max_{x \in M} |\nabla R_\lambda(x)|.
$$
Similarly, we denote by $D^\ell R_\lambda$ the $k$-th covariant derivative of $R_\lambda$.
By the standard boot-strap argument by differentiating the equation \eqref{eq:Reeb-ODE},
together with $\|R_\lambda(x)\|_g= 1$ for any triad metric $g$,
we immediately achieve the following a priori pointwise $C^{\ell +1}$ estimate of 
any solution $\gamma$ thereof.

\begin{lem}\label{lem:ODE-estimate} 
 For any solution $\gamma$ of \eqref{eq:Reeb-ODE}, there exists a 
 polynomial  $M_\ell = M_\ell(r_1, \cdots, r_\ell)$ of $\ell$ variables $r_i $ with positive coefficients of at most degree $\ell$
 depending only on $\lambda$ such that
$$
\|\nabla_t^{\ell +1} \dot \gamma\|_{C^0} \leq M_k\left(\{ \|D^i R_\lambda\|\}_{0 \leq i\leq \ell}\right).
$$
In particular, the linearization operator 
$$
\xi \mapsto \frac{D\xi}{dt} - DR_\lambda(\gamma)(\xi)
$$
is a first order linear elliptic operator on $\R$.
\end{lem}

\begin{rem} This kind of inequality, especially the usage of a polynomial $M_k$ with positive coefficients,
is a common practice in the literature that involves the bootstrap $C^k$ estimates appearing, e.g.,
in the simpleness study of the group of $C^k$ diffeomorphisms 
and in the study of a priori estimates for nonlinear elliptic partial differential equations.
(See \cite{mather,mather2}, \cite{epstein:simplicity} for the first \cite{oh-wang:CR-map1} for the latter, for example.)
The above lemma is a version of such estimates applied to the nonlinear ODE of Reeb trajectories.
\end{rem}

\subsection{Off-shell framework of the moduli space of Reeb trajectories}
\label{subsec:offshell-framework}

Motivated by the Gromov-Witten-Floer theory and Lemma \ref{lem:ODE-estimate}, we  
regard the assignment $\Upsilon_\lambda$ defined by
\be\label{eq:DC-lambda}
\Upsilon_\lambda(\gamma): =   \dot \gamma - R_\lambda(\gamma)
\ee
as a parameterized smooth section of the (infinite dimensional) vector bundle
and consider its covariant linearization.  The details are now in order.

We denote by $\CP_M^\R$ the set of smooth curves $\gamma: \R \to M$, i.e.,
\be\label{eq:CPMR}
\CP_M^\R: = \{ \gamma: \R \to M \mid \text{\rm smooth}\}.
\ee
For each given $\gamma \in \CP_M^\R$, we consider the set
$\Gamma(\gamma^*TM)$ consisting of smooth sections of the pull-back bundle $\gamma^*TM \to \R$,
and form the union
\be\label{eq:CEMR}
\CE_M^\R: = \bigcup_{\gamma \in \CP_M^\R} \{\gamma\} \times \Gamma(\gamma^*TM)
\ee
as a fiber bundle over $\CP_M^\R$.
In regard to the \emph{precise} off-shell framework of the relevant Fredholm theory, 
we mention that the section $\gamma \mapsto \dot \gamma -R_\lambda(\gamma)$ is 
a nonlinear \emph{elliptic} operator, but \emph{its domain  $\R$ is non-compact $\R$} and
\emph{we do not impose any prescribed asymptotic condition} on the Reeb trajectories, except
to the effect of the asymptotic behavior of the trajectories arising from the fact that
the target $M$ is compact. Because of no asymptotic behavior imposed,
we should always work with  compact open $C^k$
topology on $\CP_M^\R$, i.e., the \emph{weak} $C^k$ topology for $2 \leq k < \infty$  
and take that of the fiber 
$$
\CE_M^\R|_\gamma = \Gamma(\gamma^*TM)
$$
in weak $C^{k-1}$ topology. Then it is well-known that the projection
\be\label{eq:CECP}
\CE_M^\R \to \CP_M^\R
\ee
is a smooth vector bundle 
over a Banach manifold $\CP_M^\R$. (See \cite{hirsch}, \cite{kriegl-michor}.)
Furthermore, the boot-strap inequality \eqref{eq:derivative-bounds} implies that
the assignment $\Upsilon_\lambda$ defines a smooth section of the vector bundle
in the \emph{weak} $C^\infty$ topology. Furthermore its zero set is 
precisely the set of $\lambda$ Reeb trajectories. We denote the zero set by
\be\label{eq:CMR-lambda}
\CM^\R (\lambda;M) := (\Upsilon_\lambda)^{-1}(o_{\CE_M^\R}),
\ee
and its parameterized version
$$
\CM_{\mathfrak C}^\R: = \bigcup_{\lambda \in \mathfrak{C}(M)} \{\lambda\} \times \CM^\R(\lambda;M).
$$
Once this standard well-known function theoretic 
fact being mentioned,  the current setting of the first order ODE is
 \emph{analytically} much easier to handle than those of either 
pseudoholomorphic curves  \cite{oh-zhu:ajm}, \cite{wendl:super-rigidity} or
of contact instantons from \cite{oh:1jet-evaluation}. 
Therefore we will focus on the essential geometric and computational parts
of the proofs in the rest of the paper.

\subsection{Mapping transversality}

A formulation of the following mapping transversality is standard.
We introduce the parameterized version 
\be\label{eq:CEMR-C}
\CE_\mathfrak{C}^\R: =\mathfrak{C}(M) \times \CE_M^\R
 =  \bigcup_{(\lambda,\gamma) \in \mathfrak{C}(M) \times \CP_M^\R} 
\{\lambda\} \times \Gamma(\gamma^*TM)
\ee
over $\mathfrak{C}(M) \times \CP_M^\R$ of the infinite dimensional vector bundle
$\CE_M^\R \to \CP_M^\R$. (See Lemma \ref{lem:ODE-estimate} and the discussion afterwards.)
We regard the map $\Upsilon:\mathfrak{C}(M) \times \CP_M^\R \to \CE_\mathfrak{C}^\R$ given by
\be\label{eq:Upsilon}
\Upsilon(\lambda,\gamma) = \dot \gamma - R_\lambda(\gamma)
\ee
as a section of the vector bundle $\CE_\mathfrak{C}^\R$. Here we simply write
 $R_\lambda \circ \gamma$ as $R_\lambda(\gamma)$.

Using this off-shell framework, we prove the following parameterized transversality theorem.
The main purpose of proving such transversality is similar to that of the Gromov-Witten-Floer theory,
 that is,  to equip the moduli space
with the smooth structure that will be used in the study  of the 1-jet evaluation transversality
in the next subsection.

\begin{thm}\label{thm:mapping-transversality} Consider the zero set $\CM^\R_{\mathfrak{C}}$
of $\Upsilon$. Then
\begin{enumerate}
\item The parameterized moduli space $\CM^\R_{\mathfrak{C}} \subset \mathfrak{C}(M) \times \CP_M^\R$
is a smooth submanifold.
\item The projection map $\Pi: \CM^\R_{\mathfrak{C}} \to \mathfrak{C}(M)$
is a nonlinear Fredholm map of index $2n+1 = \dim M$.
\end{enumerate}
In particular, for any regular value $\lambda$ of $\Pi$, the preimage
$$
\CM^\R(\lambda;M): = \Pi^{-1}(\lambda)
$$
carries a natural smooth structure.
\end{thm}
\begin{proof} As usual, we
denote by $\frac{D}{dt} = \nabla_{\dot \gamma}$
 the covariant derivative with respect to the Levi-Civita connection $\nabla$ of the triad metric.
 
It follows from  Lemma \ref{lem:ODE-estimate} of Section \ref{sec:moduli} that 
 the map $\Upsilon_\lambda:\gamma \mapsto \dot \gamma - R_\lambda(\gamma)$ is an 
nonlinear elliptic operator, i.e., its linearization map 
$$
D\Upsilon_\lambda(\gamma)  : \Gamma_c(\gamma^*TM)  \to \Gamma_c(\gamma^*TM) 
$$
is a first-order linear ordinary differential  operator given by
$$
D\Upsilon_\lambda(\gamma) = \frac{D}{dt} - DR_\lambda(\gamma)
$$
where $\Gamma_c(\gamma^*TM)$ is the space of 
\emph{compactly supported} smooth sections of $\gamma^*TM$. 
Furthermore the zero set of $\Upsilon_\lambda$
is precisely the set of $\lambda$ Hamiltonian trajectories \eqref{eq:CMR-lambda}
$$
\CM^\R (\lambda;M) := (\Upsilon_\lambda)^{-1}\left(o_{\CE_M^\R}\right) =: \Upsilon_\lambda^{-1}(0).
$$
The following is the basis of the generic mapping transversality result.  As in the Gromov-Witten-Floer theory,
 we employ the standard argument of Fredholm alternative to prove the vanishing of  the $L^2$-cokernel.
For this purpose, we employ the \emph{triad metric} associated to the contact triads
$(M,\lambda,J)$ mentioned in Subsection \ref{subsec:triad-connection}.

We start with the following ampleness lemma.

\begin{lem} \label{lem:ampleness-W0} Assume $c = H(\gamma(0))$
is a regular value of $H$ with $H^{-1}(c) \neq \emptyset$. Then the subspace
$$
W_{\gamma(0)}: = \span_\R \left\{Y_\alpha(0) \in T_{\gamma(0)} M \, \Big|\, 
\alpha \in T_\lambda \mathfrak{C}(M) \right\}
$$
is transverse to the level set $H^{-1}(c)$. In fact, it is enough to
consider the variation $\alpha$ of the form $\alpha = h \, \lambda$ for $h \in C^\infty(M)$.
\end{lem}
\begin{proof} Recall from Section \ref{sec:phasespace} that we have  the canonical
identification $T_\lambda \mathfrak{C}(M) \cong \Omega^1(M)$.
We just consider variation of the form $\alpha = h\, \lambda$ with 
arbitrary $h \in C^\infty(M)$, and utilize the precise formula  of
 $Y_\alpha = \delta_\lambda(R_\lambda)(\alpha)$ given in \eqref{eq:Yalpha=}
$$
Y_\alpha = X_{(h;\lambda)}.
$$
It then obviously follows from this formula and  the ampleness lemma, Lemma \ref{lem:Ham-ampleness}, that
there is some $\alpha = h\, \lambda$ such that
$$
Y_\alpha(\gamma(0))  \pitchfork H^{-1}(c)
$$
This finishes the proof.
\end{proof}

We now state basic Fredholm properties of the parameterized 
linearized operator $D\Upsilon(\lambda,\gamma)$.

\begin{prop} Take the aforementioned completion of the  map
$\Upsilon : \CP_\mathfrak{C}^\R \to \CE_\mathfrak{C}^\R$  given in \eqref{eq:Upsilon} as a parameterized section
of the vector bundle $\CE_\mathfrak{C}^\R \to \CP_\mathfrak{C}^\R$.
The covariant linearization map, which is a fiberwise Fredholm operator,
$$
D\Upsilon(\lambda,\gamma): T_{(\lambda,\gamma)}\CP^\R_{\mathfrak{C}} \to \Gamma(\gamma^*TM)
$$
is surjective at all points
$$
(\lambda, \gamma) \in \Upsilon^{-1}(0) \subset \mathfrak{C}^{\text{\rm reg}}(M) \times \CP_M^\R.
$$
In particular the zero set $\Upsilon^{-1}\left(o_{\CE_{\mathfrak C}^\R}\right)$ 
is an infinite dimensional smooth manifold, which, as a set,  does not depend on the choice
of  $k \geq 2$ appearing in the $C^k$ completion employed.
\end{prop}
\begin{proof} By covariantly linearizing the expression $\dot \gamma - R_\lambda(\gamma)$
along $\delta(\lambda,\gamma) = (\alpha,\eta)$  at $(\lambda,\gamma)$, we derive
\be\label{eq:1st-variation}
D\Upsilon(\lambda,\gamma)(\alpha, \eta) = \frac{D\eta}{dt} - DR_\lambda(\eta) 
-\delta_\lambda(R_\lambda)(\alpha)
\ee
with $\eta \in \Gamma_c(\gamma^*TM)$.  

We need to analyze the effects of the last variation. We have directly
checked that the variational vector field
 \be\label{eq:Zalpha-defn}
 \delta_\lambda(R_\lambda)(\alpha)=: Y_\alpha \in \Gamma(\gamma^*TM)
 \ee
 is determined by its defining equation \eqref{eq:Yalpha}.

We now show the surjectivity of the completed operator of $D\Upsilon(\lambda,\gamma)$, 
to the given Banach completion of $\Gamma(\gamma^*TM)$
via the application of the Fredholm alternative by considering the $L^2$-cokernel
the detail of which is now in order.

We first mention that  the ellipticity of the operator $D\Upsilon_\lambda(\gamma)$,
the restriction of $D\Upsilon_\lambda(\gamma)$ to the fiber $\Pi^{-1}(\lambda)$,
proves closedness of its image 
therein. 
Suppose $Z \in \Gamma(\gamma^*TM)$ satisfies
$$
\int_\R \left\langle D\Upsilon(\lambda,\gamma)(\alpha, \eta) , Z\right\rangle \, dt = 0
$$
for all compactly supported smooth test functions $(\alpha,\eta)$. Then
the  integral equation can be split into  the following two equations:
\bea
\int_\R \left\langle \frac{D\eta}{dt} - DR_\lambda(\eta)  , Z \right\rangle \, dt & = & 0, 
\label{eq:variation-gamma}\\
\int_\R \left\langle Y_\alpha , Z\right\rangle \, dt&= & 0, \label{eq:variation-lambda}
\eea
for all compactly supported $\eta$, $\alpha$  respectively. 
From the first equation, we derive the adjoint equation
\be\label{eq:adjoint}
-\frac{DZ}{dt} - DR_\lambda(Z) = 0
\ee
by integration by parts, which is a linear first-order ODE. 
Therefore to show $Z = 0$, it is enough to show $Z(0) = 0$ at a point $0$. 

For this purpose, we now utilize the second component \eqref{eq:variation-lambda} and
the  ampleness lemma stated in Lemma \ref{lem:ampleness-W0} as follows.
We note that since $h_\alpha: M \to \R$ can be chosen arbitrarily, 
we can take $h_\alpha$ so that it is supported at  $\gamma(0)$ and that it approximates 
the Dirac-delta measure $ v \, \delta_{\gamma(0)}$ for $u \in T_{\gamma(0)}M$.
Then the integral becomes $\langle v,  Z(0)\rangle$. Furthermore Lemma \ref{lem:Ham-ampleness}
shows that we can choose any vector $v \in T_{\gamma(0)}M$. Then
the standard argument using a cut-off function proves finish the proof of $Z(0) = 0$
and hence $Z = 0$.

Since $D\Upsilon(\lambda, \gamma)$ is a closed operator, the Hahn-Banach theorem proves that
$\Image D\Upsilon(\lambda, \gamma) = \Gamma(\gamma^*TM)$, i.e., that
 $D\Upsilon(\lambda, \gamma)$ is surjective. This finishes
the proof of the proposition.
\end{proof}

Once the surjectivity stated in the above proposition is established, Theorem \ref{thm:mapping-transversality}
follows from the implicit function theorem \cite{sergeraert}.
\end{proof}

Denote by 
\be\label{eq:CC-map}
\mathfrak{C}^{\text{\rm map}}(M)
\ee
the set of regular values of the projection $\Pi$ in Theorem \ref{thm:mapping-transversality}.

\begin{rem}
\emph{Analytically}, our transversality study is a standard practice in the parameterized Fredholm analysis 
which we apply to the moduli problem of 
the Reeb trajectory equation $\dot x = R_\lambda(x)$ 
under the perturbation of contact forms $\lambda$. In this regard, the ellipticity stated in Lemma \ref{lem:ODE-estimate}
of the equation of nonlinear ODE $\dot \gamma = R_\lambda(\gamma)$ is fundamental to have a 
Fredholm completion of the linearized operator $D\Upsilon_\lambda(\gamma)$ in a suitable Banach
function space.
\end{rem}

\begin{rem}\label{rem:Taubes} As well-known to the experts,
there are two ways to derive the $C^\infty$ result from those of finite $C^\ell$ results. 
The way we follow in the present paper is the so-called 
Taubes' approach.  In short, this approach is first to prove the genericity result for each finite regularity 
class $C^\ell$, say $\ell \geq 2$, and then take the countable intersection 
$$
\bigcap_{\ell \geq 2}^\infty \mathfrak{C}_\ell^{\text{\rm np}}(M) =: \mathfrak{C}_{\text{\rm Ta}}^{\text{\rm np}}(M).
$$
 (see  \cite{mcduff-salamon-quantum}, \cite[p.343-344]{oh:book2} for complete explanation of Taubes' approach
 in the study of pseudoholomorphic curves,  for example).
 Another approach, which is not needed  in the present paper  (but can replace Taubes' approach),
 is the so-called Floer's  $C^\varepsilon$-norm approach \cite{floer:unregularized} of directly working 
 with $C^\infty$ setting (locally in the given moduli space).
 See \cite{oh-savelyev:strict-contact} for the explanation of the relevant off-shell framework in the context of
 generic study of strict contactomorphisms following \cite{floer:unregularized}, \cite{wendl:super-rigidity}.
 \end{rem}

\section{Pointed moduli space and its evaluation maps}
\label{sec:transversality-statements}

In the present section, we will need to study the 
general  Reeb dynamics 
$$
\dot x = R_\lambda(x)
$$
associated to the contact form $\lambda$. We will
examine the intersection property of Reeb trajectories with the level set
$$
\Sigma_H(c) = H^{-1}(c)
$$
of $H$ for a regular value $c$ of $H$. \emph{Provided $H$ is non-constant}, there exists 
a regular value $c \in \R$ with $H^{-1}(c) \neq \emptyset$.  

We remark that the evaluation map
\be\label{eq:ev0}
\ev_0: \CM^\R (\lambda;M) \to M; \quad \ev_0(\gamma) := \gamma(0)
\ee
defines a one-to-one correspondence by the existence, uniqueness and continuity of 
solutions of ODE applied to the initial value problem \eqref{eq:Reeb-ODE} and its derivatives.

We also need to consider its pointed versions of the above moduli space whose definition is now in order.
Let 
$$
0 < t_1 < t_2 < \cdots < t_k  
$$
 be  $k$ marked points in $\R$. We denote the set of such points by
$$
\operatorname{Conf}_{k}(\R_+): = \{(t_1, \cdots, t_k) \in \R^{k} \mid 0 < t_1 < \cdots < t_k  \}
$$
which is an open subset of $\R^{k}$.
Then we define the set
\beastar
\CP_{M;k}^\R & = & \CP_M^\R \times \operatorname{Conf}_{k}(\R_+)\\
& = & \{ (\gamma, (t_1, \cdots, t_k)) \in \CP_M^\R \times \R^{k} \mid 0 < t_1 < \cdots < t_k\}
\eeastar
and consider the natural evaluation maps $\ev_i$ given by $\ev_i(\gamma,\vec t) = \gamma(t_i)$.

We denote by
$$
\CM_k^\R (\lambda;M)
$$
the subset of $\CP_{M;k}^\R$ consisting of the pairs 
$$
(\lambda,(t_1, \cdots, t_k))
$$
with $\gamma$ being a Reeb trajectory of $\lambda$. 
We write $\vec t = (t_1, \cdots, t_k)$ and consider the natural evaluations maps
\be\label{eq:evi}
\ev_i: \CM_k^\R (\lambda;M) \to M, \quad \ev_i(\gamma,\vec t) = \gamma(t_i).
\ee
We also consider the product evaluation map
\be
\ev_+ : =  \ev_1 \times \cdots \times \ev_k: \CM_k^\R (\lambda;M)  \to M^k. \label{eq:ev+}
\ee

Next we will need to consider the one-jet evaluation maps of $\ev_+$ 
$$
\dot{\ev}_+: \CM_k^\R(\lambda; M) \to (TM)^k
$$
where $\dot{\ev}_+$ is the derivative map defined by
\be\label{eq:1jet-evaluation}
\dot{\ev}_+ (\gamma, \vec t)  = \left(\dot \gamma(t_1) \cdots \dot \gamma(t_k)\right)
\in T_{\gamma(t_1)}M \oplus \cdots \oplus T_{\gamma(t_k)}M.
\ee
We also denote by $\dot{\ev}_i$ the individual map
\be\label{eq:1jet-evaluation-i}
\dot{\ev}_i(\gamma, \vec t)  = \dot \gamma(t_i)
 \in \gamma^*TM|_{t_i} = T_{\gamma(t_i)}M
\ee
for each $i=1, \ldots, k$. Obviously we can define $\dot{\ev}_i$ for any pair $(\gamma,\vec t)$ in $\CP_{M;k}^\R$
since $\gamma$ is assumed to be smooth.

We  form the union
\be\label{eq:MkR}
\CM_{\mathfrak{C};k}^\R := \bigcup_{\lambda \in \mathfrak{C}(M)} \{\lambda\} \times \CM_k^\R(\lambda;M).
\ee
Here $\CM_k^\R(\lambda;M)$ is the $k$-marked moduli space of the 
$\lambda$ Hamiltonian  trajectories. By definition, it is a subset of 
the trivial product bundle 
$$
\CP^\R_{\mathfrak{C};k} = \mathfrak{C}(M) \times \CP_{M;k}^\R \to \mathfrak{C}(M).
$$
We also consider the parameterized evaluation maps
$$
\Ev_i:\CM^\R_{\mathfrak{C};k} \to M
$$
defined by
 $$
 \Ev_i( \lambda,\gamma,\vec t) = \gamma(t_i), \quad \vec t = (t_1, \cdots, t_k)
  $$
and their product
 $$
  \Ev_+:\CM^\R_{\mathfrak{C};k} \to  M^k.
 $$

\section{Parameterized one-jet $H$-evaluation transversality}
\label{sec:1jet-evaluation}

In this section, we carry out a study of a parameterized 1-jet transversality 
which will be essential for our proof of the main theorem of the present paper.

The overall scheme of our transversality study of parameterized maps over the set of $\lambda$s 
appearing in the transversality statements
mimics that of various generic transversality statements
of the Gromov-Witten-Floer theory for the parameterized moduli space of solutions $(u,J)$ of  
the pseudoholomorphic curve equation $\delbar_J u = 0$ under the choice of $J$, which
one regards as the zero set of the map $\Phi(J,u): = \delbar_J u$.
(See \cite{floer:unregularized}, \cite{mcduff-salamon-quantum}, \cite{oh:book1}, for example.)
However the formulation of a proper off-shell framework that enables us to apply
the Sard-Smale theorem is not at all clear at first sight. 

\begin{rem}
Analytically much harder versions of the evaluation transversality 
 of pseudoholomorphic curves and of contact instantons 
 in the spirit of the present paper were previously given in \cite{oh-zhu:ajm,oh:1jet-evaluation}, 
 \cite{wendl:super-rigidity}, and \cite{oh:contacton-transversality,oh:1jet-evaluation} respectively.
The current case of the moduli space of solutions of (elliptic) ODE 
is a much simpler case than those therein
 which deal with nonlinear elliptic PDEs, \emph{once a proper off-shell framework}
 is found (See Section \ref{sec:wrap-up}, especially 
 Remark \ref{rem:table} and the comparison table therein.)
 \end{rem}

We introduce the main evaluation maps that 
play fundamental role in our proof of the main theorem of the present paper.

\begin{defn}\label{defn:H-evaluation}
We define the \emph{parameterized $H$-evaluation map at $t = 0$} 
to the composition map 
$$
\mathfrak{C}(M) \times \CP_M^\R \stackrel{\Ev_0}{\longrightarrow} M \stackrel{H}{\longrightarrow}  \R
$$
given by
 \be\label{eq:Ev0H}
\Ev_0^H( \lambda, \gamma) = H(\gamma(0)).
 \ee
 Similarly we define \emph{the 1-jet $H$-evaluation map} $\Ev_i^H: = dH(\dot \Ev_i)$ by
\be\label{eq:1jet-evaluationmap}
\dot{\Ev}_i^H(\lambda,\gamma, t_i) := \left(H(\gamma(t_i)), dH(\dot \gamma(t_i))\right) \in \R \times \R
\ee
for $i = 1, \cdots, k$.  
\end{defn}

We mention that when $c$ is a regular value $H^{-1}(c)$ is a smooth hypersurface,
and that if
$\dot{\Ev}_i^H(\lambda,\gamma, t_i) = (c,0)$, we have, by definition,
\be\label{eq:1jet-constraint}
\gamma(t_i) \in H^{-1}(c) \quad \& \quad  \dot \gamma(t_i) \in T(H^{-1}(c)).
\ee
We will be particularly interested in their product 
$$
\dot{\Ev}_+^H: \mathfrak C(M) \times \CP_M^\R \times \operatorname{Conf}_{k}(\R_+) \to (\R \times \R)^k
$$
defined by
\be\label{eq:dotEvHk}
\dot{\Ev}_+^H( \lambda, \gamma, \vec t) 
= \left(\dot{\Ev}_1^H(\lambda,\gamma, t_1), \cdots, \dot{\Ev}_k^H( \lambda,\gamma, t_k)\right)
\ee
for a triple $(\lambda, \gamma, \vec t)$ with $\vec t = ( t_1, \cdots, t_k)$ 
with $0 < t_1 < \cdots <t_k  $. 
%

We have the following explicit formulae of the off-shell derivatives of the $H$-evaluation maps.

\begin{lem}\label{lem:EV-derivatives} Let $\nabla$ be the Levi-Civita connection.
For each $i = 1, \cdots, k$, we have
\bea\label{eq:Daleph}
&{}&
(d\Ev_0^H) (\alpha,\eta, \vec a) =  dH(\eta(0)), \nonumber \\
&{}& (d\dot{\Ev}_+^H) (\alpha,\eta, \vec a) \nonumber \\
&{}& = \prod_{i=1}^k
\left( dH(\eta(t_i) + a_i dH(\dot \gamma(t_i)),  dH\left(\frac{D \eta}{dt}(t_i) + a_i 
 \frac{D \dot \gamma}{\del t}(t_i)\right) + \nabla_\eta(dH)(\dot \gamma(t_i))  \right) \nonumber\\
 &{}&
\eea
where we write $\nabla_{\dot \gamma} = \frac{D}{dt}$ as usual in Riemannian geometry.
\end{lem}
\begin{proof} 
For the variation of the map
$$
\Ev_0^H:   (\lambda,\gamma)\mapsto H(\gamma(0)),
$$
we compute 
$$
d\Ev_0^H(\alpha,\eta,\vec a) =
\frac{d}{d s}\Big|_{s =0} H(\gamma_s(t_i + a_i s) = dH(\eta(t_i) + a_i \dot \gamma(t_i)).
$$
For the variation of the map $\dot{\Ev}_+: (\lambda,\gamma, \vec t) \mapsto \dot \gamma(t_i)$, we compute
\beastar
d\dot{\Ev}_+^H(\alpha,\eta,\vec a) & = & \frac{d}{ds}\Big|_{s = 0} dH(\dot \gamma_s(t_i + s))\\
& = & \nabla_\eta(dH)(\dot \gamma(t_i)) + dH\left((\nabla_t \dot \gamma)(t_i) + \frac{D \eta}{d t}(t_i)\right)
\eeastar
where we use the torsion-free property of the Levi-Civita connection 
for the second equality.

By definition, a direct calculation for each $i = 1, \cdots, k$ leads to the identity \eqref{eq:Daleph}
for the variation $(\alpha,\eta,\vec a) = \delta(\lambda,\gamma, \vec t)$
which finishes the proof.
\end{proof}

We define the set of pointed moduli space of $\CM_k^\R(\lambda;M)$ 
of $\lambda$-Reeb trajectories and its subset defined by
$$
\CM_k^\R(\lambda;M,\Sigma_H(c))
 = (\ev_0^H \times \dot{\ev}_+^H)^{-1}\left(\{c\} \times (\{(c,0)\})^k \right).
 $$
We form the union
\be\label{eq:CMSigma-reg-Hk}
\CM_{\mathfrak{C};k}^{\Sigma_H(c)} : = 
\bigcup_{\lambda \in  \mathfrak{C}(M)} \{\lambda \} \times 
\CM_k^\R(\lambda; M, \Sigma_H(c))
\ee
as a parameterized moduli space of Reeb trajectories.
Then we have
\be\label{eq:CMHCk}
\CM_{\mathfrak{C};k}^{\Sigma_H(c)}
= (\Ev_0^H \times \dot{\Ev}_+^H)^{-1}\left(\{c\} \times (\{(c,0)\})^k \right).
\ee
We put their union over $c \in \R$ and write the union as
\be\label{eq:CMCk}
\CM_{\mathfrak{C};k}^{\Sigma_H} := \bigcup_{c \in \R} \CM_{\mathfrak{C};k}^{\Sigma_H(c)}.
\ee
\begin{defn} We define the open subset
\be\label{eq:circCM}
{\CM}^{\circ}_k(\lambda; M,\Sigma_H(c))
: = \{ (\gamma,\vec t) \in \CM_k(\lambda;M,\Sigma_H(c)) \mid \gamma(t_i) \, \text{\rm are distinct
for all $i$}\}.
\ee
We define ${\CM}_{\mathfrak{C};k}^{\circ; \Sigma_H(c)}$ and 
${\CM}_{\mathfrak{C};k}^{\circ;\Sigma_H}$ accordingly.
\end{defn}

\begin{rem}\label{rem:CMcircle} It is  worthwhile to mention that on any compact contact manifold there is
a positive gap $T_{(M,\lambda)} > 0$ such that there is no closed orbit of period
$0 < T < T_{(M,\lambda)}$ so that $\gamma(t_i) \neq \gamma(t_j)$ for $i \neq j$ automatically
holds. Therefore we have
\beastar
&{}&  \{ (\gamma,\vec t) \in {\CM}_k^{\circ,\R}(\lambda; M,\Sigma_H(c))
  \mid 0 <t_1 < \cdots < t_k < T_{(M,\lambda)} \} \\
& = & \{ (\gamma,\vec t) \in \CM_k^\R(\lambda; M,\Sigma_H(c)) \mid  0 <t_1 < \cdots < t_k < T_{(M,\lambda)} \}.
\eeastar
\end{rem}

Next we define the map 
\be\label{eq:alephkH}
\Upsilon^{(k)}_H: \mathfrak{C}(M) \times \CP_M^\R \times \operatorname{Conf}_{k}(\R_+) \to 
\CE_{\mathfrak C}^\R(M) \times \R \times (\R \times \R)^k
\ee
to be
\bea\label{eq:alephk}
\Upsilon^{(k)}_H \left( \lambda, \gamma, \vec t\right) & = & \left(\Upsilon(\lambda,\gamma), 
H(\gamma(0)), 
\dot{\Ev}_+^H(\lambda, \gamma, \vec t)\right) \nonumber\\
& = & \left(\Upsilon(\lambda,\gamma),H(\gamma(0)), \left(
\dot{\Ev}_1^H(\lambda,\gamma, t_1), \cdots, \dot{\Ev}_k^H(\lambda,\gamma, t_k)\right) \right).
\nonumber\\
&{}& 
\eea
Then we have
$$
 \left(\Upsilon_H^{(k)}\right)^{-1}(\{(0)\} \times \{c\} \times \{(c,0)\}^k) = 
 \CM_{\mathfrak{C};k}^{\Sigma_H}
 $$
 by definition. Finally we define  the map 
 \be\label{eq:aleph(k)}
 \aleph^{(k)}: \mathfrak{C}(M) \times \CP_M^\R \times 
C^\infty(M) \times \operatorname{Conf}_{k}(\R_+) \to 
\CE_{\mathfrak C}^\R(M) \times \R \times (\R \times \R)^k 
\ee
by $ \aleph^{(k)}(\lambda,\gamma, H, \vec t): = \Upsilon_H(\lambda, \gamma,\vec t)$. 

\emph{Here and hereafter, we often write the zero section of the vector bundle
$\CE_{\mathfrak C}^\R(M)  \to \CP_{\mathfrak C}^\R(M)$ or of
$\CE_M^\R \to  \CP_M^\R$ just $\{0\}$.}

 \begin{prop}\label{prop:1jet-evaluation} 
Let  $(\lambda, \gamma, H, \vec t) \in \CM_{\mathfrak{C};k}^{\Sigma_H}$ for
$$
\vec t = ( t_1, \cdots, t_k), \quad 0 < t_1 < \cdots <t_k. 
$$
Then the  map 
$$
\aleph^{(k)}: \mathfrak{C}(M) \times \CP_M^\R \times 
C^\infty(M) \times \operatorname{Conf}_{k}(\R_+) \to 
\CE_{\mathfrak C}^\R(M) \times \R \times (\R \times \R)^k
$$
is transverse to 
$$
\{0\} \times \{c\} \times \{(c,0)\}^k  
\subset\CE_{\mathfrak C}^\R(M) \times 
\R \times (\R \times \R)^k
$$
 for any non-zero regular value $c$ of non-constant $H$ and $(\gamma, \vec t) \in {\overset{\circ}{\CM}}_k(M,\Sigma_H(c))$.
\end{prop}
\begin{proof} 
At a point $(\lambda, \gamma, \vec t)$ satisfying
$$
(\lambda, \gamma, \vec t) \in \left(\Upsilon_H^{(k)}\right)^{-1} \left(o_{\CE_{\mathfrak C}^\R(M)} \times 
\{c\} \times \left(\{c\} \times T\Sigma_H(c)\right)^k \right),
$$
we obtain its linearization
\beastar
&{}&
D\Upsilon_H^{(k)}(\lambda,\gamma,\vec t)(\alpha,\eta, \vec a)\\
& = & 
\Big(D\Upsilon(\lambda,\gamma)(\alpha, \eta),d\Ev_0^H(\alpha,\eta,\vec a), 
d\dot{\Ev}_+^H(\alpha,\eta,\vec a)\Big).
\eeastar
We have the following explicit formula for the off-shell derivative.

Now let $c$ be a regular value of $H$. 
By Lemma \ref{lem:Ham-ampleness}, the map $\Ev_0^H(\lambda,\gamma)$ is transverse to
$\Sigma_H(c) = H^{-1}(c)$.

For the map $\dot \Ev_+^H$, each factor of the second formula
in  \eqref{eq:Daleph}  is reduced to
$$
\left(dH(\dot \eta(t_i)),  dH\left(\frac{D \eta}{dt}(t_i) \right) + \nabla_\eta(dH)(\dot \gamma(t_i))\right)
 $$
 for the variation $(\alpha,\eta,h,\vec a) = \delta(\lambda,\gamma, H,\vec t)$ 
  at 
 $(\lambda,\gamma, H,\vec t) \in \CM_{\mathfrak{C};k}^{\Sigma_H}$ because 
 $$
 \frac{D \dot \gamma}{dt}(t) = \nabla_{\dot \gamma(t)} R_\lambda = 
 \left(\nabla_{R_\lambda}R_\lambda\right)(\gamma(t)) = 0
 $$
 recalling, from Lemma \ref{lem:Rlambda-killing},
  $\nabla_{R_\lambda}R_\lambda = 0$ with respect to the Levi-Civita connection  of the triad metric.
 (One may also use the contact triad connection.) 
 
 As before we have only to consider the variation $\alpha$ of the form  $h_\alpha\, \lambda$
 with $h_\alpha \in C^\infty(M,\R)$.
 It follows from the above explicit formula and the regularity of $H$ on $H^{-1}(c)$ that for each $i$ the map
$$
 (\alpha,\eta,h, \vec a) \mapsto \left (dH\left(\frac{D \eta}{dt}(0) \right), dH\left(\frac{D \eta}{dt}(t_i) \right)
 + \nabla_\eta(dH)(\dot \gamma(t_i))\right)
 $$
is transverse to $\Sigma_H(c)$  by Lemma \ref{lem:ampleness-W0} and
Corollary \ref{cor:Yalpha-small}. We note that since $\eta$ can be chosen arbitrarily, 
we may choose it so that $\eta(t_i) = 0$ but $dH(\frac{D\eta}{dt})$ to be
nonzero at  $t= 0$ and $t = t_i$ so that it approximates 
the Dirac-delta measures $u_0 \delta_{0}$, and  $u_i \delta_{t_i}$  
for an arbitrary tuple $\vec u= \{u_i\}_{i=1}^k$  
of constants $u_i$.

\emph{By the standing hyppotheses $dH(\gamma(t_i)) \neq 0$ and $\{\gamma(t_i)\}_{i=1}^k$ are all distinct}, 
 we can choose $h_\alpha$ so that it is supported near 
$\{\gamma(t_1), \cdots, \gamma(t_k)\}$ and approximates the product delta measure
$$
\prod_{i=1}^k u_i \delta_{\gamma(t_i)} \sim h_\alpha
$$
in the $L^2$ sense for any given tuple $\{u_i \in T_{\gamma(t_i)}M\}_i$ of vectors. 
Therefore we can choose $(\eta, \alpha = h_\alpha\,  \lambda)$ so that the terms $dH(\eta(t_i))) \neq 0$ and
$$
dH \left(\frac{D \eta}{dt}(t_i)\right) + \nabla_\eta(dH)(\dot \gamma(t_i))
$$
can be made non-vanishing.

Therefore for any element satisfying 
$(\gamma,\vec t)  \in  \CM^{\circ}_k(\lambda;M,\Sigma_H(c))$ i.e., for  those 
$$
(\lambda,\gamma,H, \vec t) \in \left(\aleph^{(k)}\right)^{-1}(\{0\} \times \{c\} \times (\{(c,0)\}^k),
$$
 we conclude that  $D\Upsilon^{(k)}(\lambda, H)$ is transverse to 
$$
\{0\} \times \Sigma_H(c) \times \left(\Sigma_H(c) \times_{\Sigma_H(c)} T\Sigma_H(c)\right)^k.
$$
Combining the above discussion, we have proved that the map $\aleph^{(k)}$ is a submersion
at any point $(\lambda,\gamma,H, \vec t) \in \left(\aleph^{(k)}\right)^{-1}(\{0\} \times \{c\} \times (\{(c,0)\}^k)$.
This  finishes the proof of Proposition \ref{prop:1jet-evaluation}.
\end{proof}

\section{Wrap-up of the proof of Theorem \ref{thm:nonprojectable-intro}}
\label{sec:wrap-up}

In this section, combining all the results obtained in the previous sections and \emph{re-packaging the
off-shell framework} suitably, we wrap-up the proof of Theorem \ref{thm:nonprojectable-intro}. 

In the previous section, we regard the pair $(\lambda,H)$ as a \emph{parameter} for
the moduli problem for the variable $(\gamma, \vec t)$ given by
\be\label{eq:moduli-problem}
\dot \gamma = R_\lambda(\gamma), \quad H(\gamma(0)) = c, \, dH(\gamma(t_i)) = 0
\ee
for a regular value $c$ of $H$.

In the present section, we need to re-package the map $\aleph^{(k)}$
introduced in \eqref{eq:aleph(k)} \emph{by taking $H$ as the main variable} for the same moduli problem
over the parameter space $\gamma\in \mathfrak{C}$, and $(\gamma,\vec t)$ 
as the \emph{system variables describing the moduli problem.}  This re-packaging is
needed to achieve the generic transversality in terms of the choice of contact forms
$\lambda$ only, not involving the choice of Hamiltonian $H$ as the statement of 
main theorems, Theorem \ref{thm:nonprojectable-intro} and Theorem \ref{thm:nonprojectable-small-intro}.

\subsection{Re-packaging of the off-shell framework}

We start with the vector bundle 
$$
\pi: \CE_M^\R \times \R \to \CP_M^\R
$$
and  consider its section space $\Gamma\left(\CE_M^\R \times \R\right)$. 
For each given $H$, we define a section 
\be\label{eq:FM}
\aleph_\lambda^{(0)}(H) \in \Gamma\left(\CE_M^\R \times \R\right) = : \mathfrak{F}(M)
\ee
 by its value $\aleph_\lambda^{(0)}(H)(\gamma)$ against $\gamma \in \CP_M^\R$ given by
\be\label{eq:aleph0H}
\aleph^{(0)}_\lambda(H)(\gamma) = (\dot \gamma - R_\lambda (\gamma), H(\gamma(0))).
\ee
We regard  $\aleph^{(0)}_\lambda: H \mapsto \aleph_\lambda^{(0)}(H)$ as a map
$C^\infty(M,\R) \to \mathfrak{F}(M)$.
Similarly we  define its pointed version  $\aleph^{(k)}_\lambda$
to be the map 
$$
\aleph^{(k)}_{\lambda}(H) = \Upsilon^{(k)}_\lambda \times_{\CP_M^\R} (\ev_0^H \times \ev_+^H)
$$
where we recall the definitions
$$
\CP_{M,k}^\R = \CP_M^\R \times \operatorname{Conf}_{k}(\R_+)
$$
and of  the section map
$$
\Upsilon_\lambda^{(k)}:  \CP_{M,k}^\R \to \CE_{M,k}^\R; \quad \Upsilon_\lambda(\gamma) = \dot \gamma - R_\lambda(\gamma)
$$
for the vector bundle $\CE_{M,k}^\R \to \CP_{M,k}^\R$,  and of the $H$-evaluation maps
$$
\ev_0^H \times \ev_+^H:  \CP_M^\R \times \operatorname{Conf}_{k}(\R_+) 
\to \R \times (\R \times \R)^k = \R^{2k+1}
$$
from Definition \ref{defn:H-evaluation}. More explicitly, we have
\be\label{eq:alephlambdak}
 \aleph^{(k)}_\lambda(H)(\gamma,\vec t) =  \big(\dot \gamma - R_\lambda(\gamma), H(\gamma(0)), \left(
dH(\dot \gamma(t_1)), \cdots,  dH(\dot \gamma(t_k)\right) \big).
\ee
By definition,  the value $\aleph^{(k)}_\lambda(H)$ lies in
\be\label{eq:CFk}
\text{\rm Map}\left(
\CP_M^\R \times \operatorname{Conf}_{k}(\R_+), \CE_M^\R 
 \times \R \times (\R \times \R)^k \right) =: \mathfrak{F}_k(M).
\ee
We have now identified the domain and codomain of the map $\aleph^{(k)}_\lambda$ as
$$
\aleph^{(k)}_\lambda: C^\infty(M,\R) \to \mathfrak{F}_k(M).
$$
Then we define the map $\aleph^{(k)}$ to be  its parameterized version
$$
\aleph^{(k)}: \mathfrak{C}(M) \times C^\infty(M,\R) \to \mathfrak{F}_k(M).
$$
By definition, we have the equality
$$
\aleph^{(k)} \left(\lambda, \gamma, H, \vec t\right)
= \aleph^{(k)}_\lambda(H)(\gamma,\vec t).
$$

We note that $\mathfrak{F}_k(M)$ is a fibration over the section space $\Gamma(\CE_M^\R \times \R) =  \mathfrak{F}(M)$.
Then we consider the following decompositions
\be\label{eq:aleph0}
\aleph^{(0)} = \bigcup_{\lambda \in \mathfrak{C}(M)} \{ \lambda\} \times \aleph^{(0)}_{\lambda},
\ee
and
\be\label{eq:alephk}
\aleph^{(k)} = \bigcup_{\lambda \in \mathfrak{C}(M)} \{ \lambda\} \times \aleph^{(k)}_{\lambda}.
\ee
By definition, we have 
 \be\label{eq:alephklambdaH}
  \aleph^{(k)}(\lambda,H) =   
\aleph^{(k)}_{\lambda}(H) = \Upsilon_\lambda \times (\ev_0^H \times \ev_+^H)
\ee
as a map from $\CP_M^\R \times \operatorname{Conf}_{k}(\R_+)$ to
$ \CE_{M}^\R \times \R \times (\R \times \R)^k$.

We take the canonical $C^{\ell+1}$-completion of $\CP_M^\R$ and $C^\ell$-completion of the fiber of 
$\CE_M^\R \to \CP_M^\R$ with respect to which the linearization of $\Upsilon_\lambda(\gamma)$ is 
a Fredholm operator.
We will establish that the linearization
map has its Fredholm index negative when $k \geq 2n+2$.

\begin{rem}\label{rem:table} It would be instructive if the readers compare the map $\aleph_\lambda^{(k)}$
with the nonlinear Cauchy-Riemann operator 
$$
\delbar_J: u \mapsto \delbar_Ju
$$
in the Gromov-Witten theory. The map $u$ corresponds to the function
$H$, $\delbar_J u$ to $\aleph^{(0)}_{H;\lambda}$
$J$ to $\lambda$, and
$$
\delbar: (J,u) \mapsto \delbar_J u
$$
to the map $\aleph^{(0)}: (\lambda, H) \mapsto \aleph^{(0)}_{H;\lambda}$.
The following table describes the relevant correspondences between the two:
\medskip

\begin{center} \Large
\begin{tabular}{|c||c|} 
\hline
\text{Gromov-Witten theory}  & \text{Moduli theory of Reeb dynamics}  \\
\hline
$u: \Sigma \to M$ &  $H: M \to \R$ \\ \hline 
$u$ somewhere injective & $H$ non-constant \\ \hline
$J \in \CJ(M)$ & $\lambda \in \mathfrak{C}(M)$  or $\mathfrak{C}(M,\xi)$\\ \hline
$\delbar_J$ & $\aleph^{(0)}_{\lambda}$ \\ \hline 
$\delbar_J u$ & $\aleph^{(0)}_{H;\lambda}$ \\ \hline
$\delbar$ & $\aleph^{(0)}$ \\ \hline
$\CM(M,J)$ & $\CM^\R(\lambda;M)$ \\ \hline
\end{tabular}
\end{center}

\end{rem}
\medskip

\subsection{Fredholm property}

In the above regard, the following is the first  fundamental step towards the proof
of Theorem \ref{thm:nonprojectable-intro}.

\begin{prop}[Fredholm property] Take the $C^\ell$-completion of the map
\eqref{eq:aleph0} 
$$
\aleph^{(0)}_\lambda: C^\infty(M,\R) \to  \mathfrak{F}(M).
$$
Then the derivative $D\aleph^{(0)}_\lambda(H)$ is a Fredholm operator of index 0 at each 
non-constant   $H$.
\end{prop}
\begin{proof} Since $H$ is non-constant, there is a regular value $c$ of $H$ 
such that $H^{-1}(c) \neq \emptyset$.

Let $(H,\gamma) \in \left(\aleph_\lambda^{(0)}\right)^{-1}(\{0\} \times \{c\})$.
By definition, this is equivalent to assuming 
 $\gamma\in \CM(\lambda; M,\Sigma_H(c))$, in particular 
$\gamma(0) \in \Sigma_H(c)$. Recalling the definition
$\mathfrak{F}(M) = \Gamma(\CE^\R_M \times \R)$
and 
$$
\aleph^{(0)}_\lambda(H)(\gamma) = (\dot \gamma - R_\lambda (\gamma), H(\gamma(0))),
$$
from \eqref{eq:aleph0H},
we also write
$$
\aleph^{(0)}_\lambda(H)(\gamma)  = \aleph^{(0)}_\lambda(H, \gamma)
$$
by a slight abuse of notation.
Then we compute its linearization
$$
D\aleph^{(0)}_\lambda(H,\gamma)(h,\eta) = \left(\frac{D \eta}{dt} - DR_\lambda(\gamma) \eta, h(\gamma(0)) + dH(\eta(0))\right).
$$
Since the differential ODE operator $\frac{D}{dt} - DR_\lambda(\gamma)$ is elliptic,
the closedness of the image is apparent.

The kernel of the linear operator
$$
D\aleph_\lambda^{(0)}(H): C^\infty(M) \to T_{\aleph_\lambda^{(0)}(H)} \mathfrak{F}(M)
$$
is given by the set of $h$'s that satisfy
\be\label{eq:kernel-eq}
\begin{cases}
\frac{D \eta}{dt} - DR_\lambda(\gamma) \eta = 0,\\
 h(\gamma(0)) + dH(\eta(0)) = 0
\end{cases}
\ee
\emph{for all $\gamma$ satisfying $\dot \gamma = R_\lambda(\gamma)$} and $\eta \in \Gamma(\gamma^*TM)$.
Since $\eta$ satisfies the linear first-order ODE 
$$
\frac{D \eta}{dt} - DR_\lambda(\gamma)\, \eta = 0, 
$$
the initial condition $\eta(0)$
uniquely determines $\eta$ for given $\gamma$.
Furthermore  $h$ satisfies
$$
h(\gamma(0)) = - dH(\eta(0)); \quad \eta(0) \in T_{\gamma(0)}M
$$ 
by the second equation of \eqref{eq:kernel-eq},
and hence $h(\gamma(0))$ is uniquely determined by $\eta(0)$. In addition,
 the evaluation map
$$
\ev_0: \CM(M,\lambda) \to M; \quad \gamma \mapsto \gamma(0)
$$
is surjective. Therefore $h$ itself is uniquely determined by $H$ by varying $\gamma$.
Combining the above, we derive
\be\label{eq:dim-ker}
\dim \left(\ker D\aleph_\lambda^{(0)}(H)\right) = \dim M = 2n+1 < \infty.
\ee

To compute the cokernel of $D\aleph_\lambda^{(0)}(H)$, we recall
$$
T_{(\gamma, c)} \mathfrak{F}(M) 
= \left\{(\xi, a)  \mid \xi \in \Gamma(\gamma^*TM), \, a \in \R \right\}
$$
by definition of $\CE^\R_M =  \bigcup_{\gamma \in \CP^\R_M} \{\gamma\} \times \Gamma(\gamma^*TM)$
where $c = H(\gamma(0))$. With respect to the obvious $L^2$ inner product on the tangent space, 
we do integration by parts and derive that the kernel element $(\zeta,b)$ of
 the $L^2$ adjoint of \eqref{eq:kernel-eq}  \emph{at each $\gamma$} satisfies
\be\label{eq:cokder-eq}
\int_M \left\langle \frac{D \eta}{dt} - DR_\lambda(\gamma) \, \eta, \zeta \right\rangle \, d\mu_\lambda + 
 b( h(\gamma(0)) + dH(\eta(0))) = 0
 \ee
for all $\eta$ and $h$.  By considering $\eta = 0$, the equation is reduced to
$$
b( h(\gamma(0)) = 0
$$
 for all $h$. Therefore we obtain $b = 0$.
Making back substitution of $b = 0$ into the above equation and doing integration by parts, we derive that
$\zeta$ satisfies the adjoint equation
$$
- \frac{D \zeta}{dt} - DR_\lambda(\gamma) \, \zeta = 0.
$$
By the same argument applied to the kernel case above, we derive
\be\label{eq:dim-cokernel}
\dim \left(\coker D\aleph_\lambda^{(0)}(H)\right) = \dim M = 2n+1 < \infty.
\ee
This finishes the proof.
\end{proof}

\subsection{Application of Sard-Smale theorem}
\label{subsec:Sard-Smale}

We now consider the parameterized map 
$$
\aleph^{(k)}: \mathfrak{C}(M) \times C^\infty(M,\R) \to \mathfrak{F}_k(M).
$$
Recall the transversality result proved in Proposition \ref{prop:1jet-evaluation} 
viewed in the current new packaging of
the off-shell framework. Then we derive by the implicit function theorem that
$$ 
{\CM}_{\mathfrak{C}^{\text{\rm map}};k}^{\circ;\Sigma_H(c)} \subset  \mathfrak{C}(M) \times C^\infty(M,\R)
$$
is a smooth Frechet submanifold of $ \mathfrak{C}(M) \times C^\infty(M,\R)$.

As before in Section \ref{sec:moduli}, we take a suitable off-shell Banach completions of
 relevant infinite dimensional Frechet manifolds,
the domain and codomain of the map $\Upsilon$,  so that the relevant  projection 
$$
\Pi_k :  
{\CM}_{\mathfrak{C}^{\text{\rm map}};k}^{\circ;\Sigma_H(c)} \to 
\mathfrak{C^{\text{\rm map}}}(M); \quad (\lambda,H) \mapsto \lambda
$$
becomes  a (nonlinear) Fredholm map on the $C^\ell$  completions of its domain and
codomain as before. In regard to Remark \ref{rem:Taubes}, we also recall that the moduli space $\Upsilon_\lambda^{-1}(0)$
(as a set)  is independent of this regularity of the completion, 
and so are other moduli spaces such as
 $$
{\CM}_{\mathfrak{C}^{\text{\rm map}};k}^{\circ;\Sigma_H(c)}
\subset  \mathfrak{C}^{\text{\rm map}}(M) \times C^\infty(M,\R)
$$  
 
\begin{thm} \label{thm:negative-index} 
Consider the projection map
$$
\Pi^{\text{\rm map}}_k: 
{\CM}_{\mathfrak{C}^{\text{\rm map}};k}^{\circ;\Sigma_H(c)}  \to \mathfrak{C}^{\text{\rm map}}(M)
$$
which is the restriction of the natural projection $\CE_k^\R(M) \times \mathfrak{C}(M) \times 
\operatorname{Conf}_k(\R_+) \to \mathfrak{C}(M)$.
We take suitable Banach completions of the domain and codomain and extend the map $\Pi^{\text{\rm map}}$
as the completion as before. Then we have the following:
\begin{enumerate}
\item $\Pi_{k}^{\text{\rm map}}$ is a (nonlinear) Fredholm map of index $2n+1 - k$ near
any nonconstant $H$ with $c$ be a regular value of $H$ such that $H^{-1}(c) \neq \emptyset$.
\item For any regular value $\lambda \in \mathfrak{C}^{\text{\rm map}}$ of  $\Pi_k^{\text{\rm map}}$,  the set
$$ 
{\CM}_k^{\circ;\R}(\lambda;M, {\Sigma_H(c)}) 
 =
{\CM}_{\mathfrak{C};k}^{\circ;\Sigma_H(c)}\cap \Pi_{H,k}^{-1}(\lambda)
 $$
is empty for any $k \geq 2n+2$ for any regular value $c$ of $H$.
\end{enumerate}
\end{thm}
\begin{proof} For the Fredholm property of this projection
 and the calculation of its index, we first mention that we have already proved that
 the linearization map of the vertical component of $\Pi_0^{\text{\rm map}}$, that is,
$$
D\Upsilon_\lambda(\gamma): \Gamma(\gamma^*TM) \to \Gamma(\gamma^*(TM)
$$
is elliptic and so its $C^\ell$   Banach completion  is Fredholm
in Section \ref{sec:moduli}. The Fredholm property of
the linearization map on the pointed moduli spaces 
is an immediate consequence therefrom since the codomains of the 
evaluation maps are compact.  
 It is easy to check from this fiberwise ellipticity that the projection $d\Pi_{H}$ is a closed
operator. This will then imply the finite dimensionality of the cokernel thereof and 
the isomorphism
$$
\coker d\Pi_k^{\text{\rm map}}|_{T(\Pi_k^{\text{\rm map}})^{-1}(\lambda)} \cong \coker D\aleph_\lambda^{(k)}
$$
either directly or by quoting a general diagram chasing
argument. (See the diagram in \cite[p. 342]{oh:book2} 
the detailed explanation of for the general diagram chasing argument.) 

Now we compute the index to be
$$
2n +1 + k - k -  k = 2n+1 - k:
$$
Here the first $k$ is the number of marked points $t_i$ and the second $k$ is 
the codimension of $\left(\Sigma_H(c)\right)^{k}$ in $M^{k}$, 
and the third $k$ is the codimension of the
derivative constraints 
$$
\dot \gamma(t_i)  \in T_{\gamma(t_i)}\Sigma_H(c) \subset T_{\gamma(t_i)}M, 
\quad i=1, \cdots, k.
$$
This finishes the proof of the theorem.
\end{proof}

Now we vary them for our purpose of \emph{making
$
{\CM}_k^{\circ;\Sigma_H(c)}(\lambda;M)$ empty for the regular values $\lambda$ of
$\Pi_{H,k}^{\text{\rm map}}$.}
In the remaining section,  it is enough to fix $k=2n+2$ so that $2k+1 - (2k+2) = -1 < 0$.
We denote by 
$$
\mathfrak{C}_{2n+2}^{\text{\rm reg}}(M)
$$
the regular values of $\Pi_{H,2n+1}^{\text{\rm map}}$, and then take the intersection
\be\label{eq:CMnp}
\mathfrak{C}^{\text{\rm np}}(M): = 
 \mathfrak{C}^{\text{\rm map}}(M) \cap \mathfrak{C}_{2n+2}^{\text{\rm reg}}(M).
\ee
This intersection is what
we have been searching for and that is mentioned in Theorem \ref{thm:nonprojectable-intro}.
Obviously it is a residual subset by construction.
Then 
  \be\label{eq:CMH-empty}
{\CM}_{2n+2}^{\circ}(\lambda;M,\Sigma_H(c)) 
 \ee
is a smooth manifold of negative dimension.
Therefore we conclude
\be\label{eq:emptyset} 
{\CM}_{2n+2}^{\circ;}(\lambda;M, \Sigma_H(c))  := \Pi^{-1}(\lambda) \cap   
{\CM}_{\mathfrak{C}^{\text{\rm map}},2n+2}^{\circ;\Sigma_H(c)} = \emptyset
\ee
for all $\lambda \in \mathfrak{C}^{\text{\rm np}}(M)$  by the Sard-Smale theorem.
By construction, its implication is that
at least one of following conditions 
  \be\label{eq:g-constraint}
 H(\gamma(0))= c \quad \& \quad \dot \gamma(t_i) \in 
 T_{\gamma(t_i)}\Sigma_H(c) \quad i = 1, \cdots, 2n+2
 \ee
 should fail to hold
 for any Hamiltonian trajectories $\gamma$ of $H$.

Unravelling the above discussion, we obtain the following corollary.
\begin{cor}\label{cor:negative-index} Let $\lambda \in \mathfrak{C}^{\text{\rm np}}(M)$. 
Suppose that $H$ is non-constant and
let $c$ be a regular value of $H$ with $H^{-1}(c) \neq \emptyset$. 
Then there exists a solution $\gamma$ of $\dot x = R_\lambda(x)$ such that $\gamma(0) \in H^{-1}(c)$ 
and that for any choice of $2n+2$ points $t_1 < t_2 < \cdots < t_{2n+2}$ satisfying
$$
\gamma(t_i) \in H^{-1}(c),
$$
there is at least one $i$ such that $\dot \gamma(t_i)$ is not tangent to $H^{-1}(c)$.
\end{cor}

Now we are ready to complete the proof of Theorem \ref{thm:nonprojectable-intro}.

\begin{proof}[Wrap-up of the proof of Theorem \ref{thm:nonprojectable-intro}]
Let $\lambda \in \mathfrak{C}^{\text{\rm np}}(M)$, assume $dH \neq 0$. 
Suppose to the contrary of Theorem \ref{thm:nonprojectable-intro} that $R_\lambda[H] = 0$. 
Then for any given constant $c \in \R$ and for  Reeb trajectory $\gamma$ issued at $\Sigma_H(c)$
remains to stay in $\Sigma_H(c)$. In particular we must have
\be\label{eq:dotgamma-tangent}
\dot \gamma(t) \in T\Sigma_H(c)
\ee
for all $t \in \R$, whenever $\gamma(0) \in \Sigma_H(c)$.

On the other hand, let $c$ be a regular value of $H$ with $H^{-1}(c) \neq \emptyset$.
Recall that by definition we have a surjective map 
$$
\ev_0: \CM^\R(\lambda; M,\Sigma_H(c)) \to \Sigma_H(c)
$$
and with $ \Sigma_H(c) \neq \emptyset$.  In particular the set 
$\CM^\R(\lambda; M,\Sigma_H(c))$ is never empty.
We also note ${\CM}_{2n+2}^{\circ, \R}(M,\Sigma_H(c))$ is diffeomorphic to
$$
{\CM}^{\circ, \R}(\lambda;M,\Sigma_H(c)) \times \operatorname{Conf}_{2n+2}(\R_+).
$$
Combining these two and Remark \ref{rem:CMcircle}, 
we have showed that 
\beastar
&{}& {\CM}_{2n+2}^{\circ, \R}(\lambda;M,\Sigma_H(c)) \cap \{0 < t_1 < \cdots t_{2n+2} < T_{(\lambda,M)}\} \\
& = & \CM_{2n+2}^\R(\lambda;M,\Sigma_H(c)) \cap \{0 < t_1 < \cdots t_{2n+2} < T_{(\lambda,M)}\},
\eeastar
which is not an empty set.
Obviously this contradicts to \eqref{eq:emptyset}. Therefore by Corollary \ref{cor:negative-index},
there must be some point $x \in \Sigma_H(c)$ such that $\gamma(0) =x$ but 
$R_\lambda(\gamma(0)) \pitchfork \Sigma_H(c)$, which contradicts to \eqref{eq:dotgamma-tangent}.
Therefore we have concluded  $R_\lambda[H] \neq 0$ for any $\lambda \in \mathfrak{C}^{\text{\rm np}}(M)$.

Then by taking the contrapositive of the standing hypothesis, we conclude 
that for any $\lambda \in \mathfrak{C}^{\text{\rm np}}(M)$, if $R_\lambda[H] = 0$, then $dH = 0$.
In other words, we have
$$
\ker R_\lambda = \{\text{\rm constant functions}\}.
$$
This finally completes the proof of Theorem \ref{thm:nonprojectable-intro}.
\end{proof}

\part{On the small phase space}
\label{part:small}

This part is the small phase space analog to  Part \ref{part:big}
 with some adaptation needed for our current usage of small phase space $\mathfrak{C}(M,\xi)$ 
 and $\mathfrak{Cont}(M,\xi)$ here. 
In this part, we prove the analog Theorem \ref{thm:nonprojectable-small-intro} to Theorem \ref{thm:nonprojectable-intro} 
for $\mathfrak{C}(M,\xi)$ and $\mathfrak{Cont}(M,\xi)$ with fixed contact structure $\xi$
instead of $\mathfrak{C}(M)$ and $\mathfrak{Cont}(M)$. Our explanations will be brief
 just by indicating the necessary changes  to be made from that of the latter case  to handle the former case,
 because the proof is largely the same except our usage of small phase space $\mathfrak{C}(M,\xi)$
in place of the big phase space $\mathfrak{C}(M)$.

\section{The mapping and 1-jet evaluation transversalities: Statements}
\label{sec:transversality-small}

In this section, we always assume that $H$ is not constant, 
without further mentioning.

We recall the description of the tangent space of the subset 
$$
\mathfrak{C}(M,\xi) \subset \mathfrak{C}(M)
$$
from Proposition \ref{prop:Tlambda-xi}
$$
T_\lambda \mathfrak{C}(M,\xi) = \{h \, \lambda \mid h \in C^\infty(M,\R)\}
$$
as a submanifold $\mathfrak{C}(M)$ and the discussion given in Section \ref{sec:phasespace}.
We start with the following identification of 
this tangent space with $C^\infty(M,\R)$, recalling the discussion in 
Section \ref{sec:phasespace}. 

We have the natural inclusion
\be\label{eq:ContMxi}
\mathfrak{Cont}(M,\xi) \cong \bigcup_{\lambda \in \mathfrak{C}(M,\xi)} \{\lambda \} \times \Cont(M,\lambda) \hookrightarrow 
\Omega^1(M) \times \Cont(M,\xi)
\ee
where $\Omega^1(M) \times \Cont(M,\xi)$ does not depend on $\lambda$. 
We  form the union
\be\label{eq:CMRk}
\CM_{\mathfrak{C}(\xi);k}^\R:= \bigcup_{\lambda \in \mathfrak{C}(M,\xi)} \{\lambda\} \times \CM_k^\R(\lambda;M).
\ee
 By definition, it is a subset  of 
the trivial product bundle 
$$
\mathfrak{C}(M,\xi) \times \CP_{M;k}^\R \to \mathfrak{C}(M,\xi).
$$
We introduce another parameterized  infinite dimensional vector bundle
\be\label{eq:CEMR-C-small}
\CE_{\mathfrak{C}(\xi)}^\R: =\mathfrak{C}(M,\xi) \times \CE_M^\R
 =  \bigcup_{(\lambda,\gamma) \in \mathfrak{C}(M,\xi) \times \CP_M^\R} 
\{\gamma\} \times \Gamma(\gamma^*TM)
\ee
over $\mathfrak{C}(M,\xi) \times \CP_M^\R$, which is the small phase space analog to 
$\CE_{\mathfrak{C}}^\R$.
We regard the map
\be\label{eq:UpsilonH-small}
\Upsilon^{\text{\rm sm}}(\lambda,\gamma) = \dot \gamma - R_\lambda(\gamma)
\ee
as a section of the vector bundle $\CE_{\mathfrak{C}(\xi)}^\R$. 

\begin{prop}[Mapping transversality]\label{prop:mapping-transversality-small} Consider the zero set
of $\Upsilon^{\text{\rm sm}}$, 
$$
\CM^\R_{\mathfrak{C}(\xi)}: = (\Upsilon^{\text{\rm sm}})^{-1}(0).
$$
 Then the following holds:
\begin{enumerate}
\item The parameterized moduli space 
$\CM^\R_{\mathfrak{C}}(\xi)\subset \mathfrak{C}(M,\xi) \times \CP_M^\R$
is a smooth submanifold.
\item The projection map $\Pi:\CM^\R_{\mathfrak{C}}(\xi)\to \mathfrak{C}(M,\xi)$
is a nonlinear Fredholm map of index $2n+1 = \dim M$.
\end{enumerate}
In particular, for any regular value $\lambda$ of $\Pi$, the preimage
$$
\CM^\R(\lambda;M): = \Pi^{-1}(\lambda)
$$
is a compact manifold of dimension $2n+1$ which is diffeomorphic to $M$.
Denote by $\mathfrak{C}^{\text{\rm map}}(M,\xi)$ the set of regular values of $\Pi$.
\end{prop}

We also  consider the 1-jet $H$-evaluation maps
defined by
\be\label{eq:1jet-evaluationmap}
\dot{\Ev}_i^H(\lambda,\gamma, t_i) := \left(H(\gamma(t_i)),dH(\dot \gamma(t_i)) \right) \in \R \times \R
\ee
and will be particularly interested in their product
$$
\dot{\Ev}_+^H: \mathfrak C(M) \times \CP_M^\R \times \operatorname{Conf}_{k}(\R_+) \to (\R \times \R)^k
$$
defined by
\be\label{eq:GammaH}
\dot{\Ev}_{+,k}^H(\lambda, \gamma, \vec t) 
= \left(\dot{\Ev}_1^H(\lambda,\gamma, t_1), \cdots, \dot{\Ev}_k^H(\lambda,\gamma, t_k)\right)
\ee
for a triple $(\lambda, \gamma, \vec t)$ with $\vec t = ( t_1, \cdots, t_k)$ 
with $0 < t_1 < \cdots <t_k$. We define a subset of $\CM_k^\R(M;\lambda)$ by
$$
\CM_k^\R(M, {\Sigma_H(c)})
: = (\ev_0^H \times \dot{\ev}_+^H)^{-1}\left(\{c\}  \times \{(c,0)\}^k \right),
$$
and consider the union
$$
\Sigma_{(H;(\cdot))} = 
\bigcup_{\lambda \in \mathfrak{C}^{\text{\rm tr}}(M,\xi)} \{\lambda\} \times \Sigma_H(c). 
$$
We also  form the union
$$
\CM_{\mathfrak{C}(\xi)}^{\Sigma_H(c)}
: = \bigcup_{\lambda \in  \mathfrak{C}^{\text{\rm tr}}(M,\xi)} \{\lambda \} \times 
\CM_k^\R(M, {\Sigma_H(c)}) 
$$
similarly as in the case of big phase space. Then we have the following small phase space
counterpart of Proposition \ref{prop:1jet-evaluation}.

\begin{prop}\label{prop:1jet-evaluation-small} 
Suppose that $H$ is non-constant. Let 
$$
(\lambda, \gamma, H, \vec t) \in \left(\Upsilon^{(k)}\right)^{-1}(\{(0)\} \times \{c\} \times \{(c,0)\}^k)
$$
with $\lambda \in \mathfrak{C}(M,\xi)$  for
$\vec t = ( t_1, \cdots, t_k)$  with $0 < t_1 < \cdots <t_k$  
such that $\gamma(t_i) \neq \gamma(t_j)$ whenever $i \neq j$.
Then the  map 
$$
\Upsilon^{(k)}: \mathfrak{C}(M,\xi) \times \CP_M^\R \times 
C^\infty(M) \times \operatorname{Conf}_{k}(\R_+) \to 
\CE_{\mathfrak C(\xi)}^\R(M) \times \R \times (\R \times \R)^k
$$
is transverse to 
$$
\{0\} \times \{c\} \times \{(c,0)\}^k  
\subset\CE_{\mathfrak C(\xi)}^\R(M) \times 
\R \times (\R \times \R)^k
$$
 for any non-zero regular value $c$ of $H$.
\end{prop}

We define the subset of $\CM_k^{\circ,\R}(\lambda;M, {\Sigma_H(c)})$ of
$\CM_k^\R(\lambda;M, {\Sigma_H(c)})$
and  form the union
$$
\CM_{\mathfrak{C}(\xi)}^{\circ; \Sigma_H(c)}
: = \bigcup_{\lambda \in  \mathfrak{C}^{\text{\rm tr}}(M,\xi)} \{\lambda \} \times 
\CM_k^{\circ,\R}(\lambda; M, {\Sigma_H(c)}) 
$$
similarly as in the case of big phase space. 

\section{The mapping and 1-jet evaluation transversalities: Proofs}
\label{sec:mapping-transversality-small}
We consider the map
$$
\Upsilon_\lambda^{\text{\rm sm}}(\gamma) = \dot \gamma - R_\lambda(\gamma)
$$
as a smooth section of the (infinite dimensional) vector bundle 
$\CE_M^\R$  over $\CP_M^\R$.  
Furthermore its zero set is 
precisely the set of $\lambda$ Hamiltonian trajectories \eqref{eq:CMR-lambda}
$$
\CM^\R (\lambda;M) := (\Upsilon_\lambda^{\text{\rm sm}})^{-1}(0),
$$
and the map $\gamma \mapsto \dot \gamma - R_\lambda(\gamma)$ is a 
nonlinear elliptic operator, i.e., its linearization map 
$$
D\Upsilon_\lambda^{\text{\rm sm}}(\gamma): \Gamma_c(\gamma^*TM)  \to \Gamma_c(\gamma^*TM) 
$$
is a first-order ordinary differential  operator, where $\Gamma_c(\gamma^*TM)$ is the space of 
compactly supported smooth sections of $\gamma^*TM$.

We now provide the off-shell framework for the study of the mapping transversality
of the 1-pointed moduli space
$$
\CM^\R(\lambda;M)
$$
under the perturbation of $\lambda$.
Then we form the parameterized bundle
$$
\CE^\R_{\mathfrak{C}(\xi)} : = \bigcup_{\lambda \in \mathfrak{C}(M,\xi)} \{{\lambda}\} \times \CE_M^\R
= \bigcup_{(\lambda,\gamma) \in \mathfrak{C}(M,\xi) \times \CP_M^\R} \{(\lambda,\gamma)\} \times \Gamma( \gamma^*TM)
$$
and consider the parameterized section $\Upsilon^{\text{\rm sm}}$  defined by 
$$
\Upsilon^{\text{\rm sm}}(\lambda,\gamma) = \Upsilon_\lambda(\gamma).
$$

The following is the basis of the generic mapping transversality result. 

\begin{prop} 
The covariant linearization map 
$$
D\Upsilon^{\text{\rm sm}}(\lambda,\gamma): T_{(\lambda,\gamma)}\CP^\R_{\mathfrak{C(\xi)}} \to \Gamma(\gamma^*TM)
$$
is surjective at all points
$$
(\lambda, \gamma) \in (\Upsilon^{\text{\rm sm}})^{-1}(0) \subset \mathfrak{C}^{\text{\rm map}}(M,\xi) \times \CP_M^\R.
$$
In particular the zero set $(\Upsilon^{\text{\rm sm}})^{-1}\left(o_{\CE_{\mathfrak C}^\R(M,\xi)}\right)$ 
is an infinite dimensional smooth manifold.
\end{prop}
\begin{proof} By covariantly linearizing the expression $\dot \gamma - R_\lambda(\gamma)$
along $\delta(\lambda,\gamma) = (\alpha,\eta)$  at $(\lambda,\gamma)$ of the form
\be\label{eq:alpha-small}
\alpha = h\, \lambda,
\ee 
we derive
$$
D\Upsilon^{\text{\rm sm}}(\lambda,\gamma)(\alpha, \eta) = \frac{D\eta}{dt} - DR_\lambda(\eta) 
-\delta_\lambda(R_{\lambda})(\alpha)
$$
with $\eta \in \Gamma(\gamma^*TM)$ as in \eqref{eq:1st-variation}.

We now recall the formula
$$
\delta_\lambda(R_\lambda)(\alpha)= X_{(h;\lambda)}
$$
from Corollary \ref{cor:Yalpha-small}.
Once we have this formula, the rest of the proof is the same as that of
Theorem  \ref{thm:mapping-transversality} and so omitted.
\end{proof}

Now we consider the pointed moduli spaces and the evaluation and 1-jet evaluation maps.
 We then consider the small phase space analogues of
 the parameterized evaluation map $\Ev_0^H$ at $t = 0$ defined in \eqref{eq:Ev0H}
 and of the \emph{augmented 1-jet evaluation map} $\dot{\Ev}_i^H$ given in \eqref{eq:1jet-evaluationmap} 
for $i = 1, \cdots, k$, that is, 
$$
\dot{\Ev}_+^H(\lambda, \gamma, \vec t) 
= \left(\dot{\Ev}_1^H(\lambda,\gamma, t_1), \cdots, \dot{\Ev}_k^H(\lambda,\gamma, t_k)\right)
$$
for a triple $(\lambda, \gamma, \vec t)$ with $\vec t = ( t_1, \cdots, t_k)$ 
with $0 < t_1 < \cdots <t_k  $  \emph{such that $\gamma(t_i) \neq \gamma(t_j)$ whenever $i \neq j$.}

An examination of the proof of evaluation and 1-jet evaluation transversalities
given in Section \ref{sec:1jet-evaluation} shows that 
once Corollary \ref{cor:Yalpha-small} is provided, all the proofs therein
can be duplicated and so establishes the 1-jet evaluation transversality
stated in Proposition \ref{prop:1jet-evaluation}.  The proofs are the same 
except the change $\CM_k^\R(M)$ to $\CM_k^\R(M,\xi)$.

\section{Wrap-up of the proof of Theorem \ref{thm:nonprojectable-small-intro}}
\label{sec:wrapup-small}

In this section, combining all the results obtained in the previous sections of Part \ref{part:small}, 
we wrap-up the proof of Theorem \ref{thm:nonprojectable-small-intro}, which we restate here.

\begin{thm}\label{thm:nonprojectable-small} Let $(M,\xi)$ be a coorientable contact manifold. Then
there exists a residual subset
 $$
 \mathfrak{C}^{\text{\rm np}} (M,\xi) \subset \mathfrak{C}(M,\xi)
 $$
consisting of $\lambda$ for which  we have
$$
\{H \mid dH \not \equiv 0 \,\, \& \,\,  \, R_\lambda[H] = 0\} = \emptyset.
$$
We call any such contact form \emph{non-projectible in small phase space}.
\end{thm}

We also have the following small phase space counterpart of Theorem \ref{thm:negative-index},
the proof of which is again the same with the replacement of the big phase space by the
small phase space.

\begin{thm}[Fredholm property] \label{thm:negative-index-small} 
Let $H$ be a non-constant autonomous Hamiltonian
and $g_{(H;\lambda)}$ the conformal exponent of $(\lambda,H)$.
 Consider the projection map
$$
\Pi_k:\CM_{\mathfrak{C}^{\circ;\text{\rm map}}(\xi);k}^{\Sigma(c)} \to \mathfrak{C}^{\text{\rm map}}(M,\xi).
$$
Then we have the following:
\begin{enumerate}
\item $\Pi_{H,k}$ is a (nonlinear) Fredholm map of index $2n+1 - k$. In particular, the index is negative
provided $k \geq 2n+2$.
\item For any regular value of  $\Pi_{H,k}$,  the set
$$
\CM_k^{\circ,\R}(\lambda; M, {\Sigma_H(c)}) 
=  \CM_{\mathfrak{C}^{\text{\rm map}}(\xi),k}^{\circ,\R} \cap \Pi_{H,k}^{-1}(\lambda)
 $$
is empty for any $k \geq 2n+2$.
\end{enumerate}
\end{thm}

As before in the big phase space, we define  a dense open subset
$$
\mathfrak{C}_{2n+2}^{\text{\rm reg}}(M,\xi)
$$
to be the set of regular values of $\Pi_k$ above.  Then we obtain 
\be\label{eq:CMnp-small}
 \mathfrak{C}^{\text{\rm np}}(M,\xi) =  \mathfrak{C}^{\text{\rm reg}_{2n+2}}(M,\xi) 
 \cap  \mathfrak{C}^{\text{\rm map}}(M,\xi).
\ee
This subset $\mathfrak{C}^{\text{\rm np}}(M,\xi) \subset \mathfrak{C}(M,\xi)$ 
is what we are searching for and that is mentioned in Theorem \ref{thm:nonprojectable-intro}.
Obviously it is a residual subset by construction.

Now the rest of the proof of Theorem \ref{thm:nonprojectable-small-intro} is the same as that of
Theorem \ref{thm:nonprojectable-intro} with $\mathfrak{C}(M)$ replaced by $\mathfrak{C}(M,\xi)$
and so omitted. This finishes the proof of Theorem \ref{thm:nonprojectable-small-intro}.

\appendix

\section{Proof of Lemma \ref{lem:Tlambda}}

By the standing hypothesis that $M$ is \emph{contactible}, there is at least one
contact structure $\xi$. We will describe the space of all nearby contact structures of $\xi$.

We start with the notion of \emph{contact vector space} summarizing the exposition of
 \cite[p. 133-134]{loose} with slight change of notations.

\begin{defn}[Contact vector space]  A \emph{contact vector space} is a triple $(V,H,\Omega)$ such that
 $V$ is a vector space equipped with \emph{contact structure} which 
is a pair $(H, \Omega)$ such that $H \subset V$ is a codimension 1 subspace and a bilinear form
$$
\Omega: H \times H \to V/H.
$$
\end{defn}

For the standard contact vector space 
$$
(V_0, H_0, \Omega_0): = (\R^{2n+1}, \R^{2n} \times \{0\}, \omega_0 \oplus 0)
$$
withe the identification $\R^{2n+1}/\R^{2n} \times \{0\} = \R$, 
the \emph{contact group} $\Cont_{2n+1}(\R)$ is the subgroup
$$
\{A \in GL(\R^{2n+1}) \mid A(H \times \{0\}) = H \times \{0\}, \, A|_H \in Sp(R^{2n} \times \{0\})\}.
$$
Consider the Grassmannian $\operatorname{Gr}_{2n}(\R^{2n+1})$  consisting of 
subspaces of codimension 1. Then any nearby contact vector structures of $(V_0, H_0, \Omega_0)$ 
on $\R^{2n+1}$ is given by pair $(H,\Omega)$ such that $H$ is a codimension 1 subspace 
near to $H_0 \in GL(\R^{2n+1})$, and a symplectic bilinear form $\Omega$ thereon.

For each $H$, we denote by $\Symp(H)$ the set of symmetric bilnear forms on $H$
equipped with the natrual topology thereon. We then form the union
$$
\mathfrak{C}(\R^{2n+1}): =\bigcup_{H \in   GL(\R^{2n+1}} \{H\} \times \Symp(H)
$$
which is a fiber bundle over $GL(\R^{2n+1})$. It carries a natural topology induced by that of
$GL(\R^{2n+1})$ and $\Symp(H)$ such that the projection 
$\pi: \mathfrak{C}(\R^{2n+1}) \to GL(\R^{2n+1})$ is continuous. By the same token, we
equip them with obvious smooth structures so that $\pi$ becomes a smooth fibration.
This being said the tangent space 
$T_{(H_0,\Omega_0)}\mathfrak{C}(\R^{2n+1})$ or the space of infinitesimal
deformation space can be identified with
$$
(\R^{2n})^* \times \text{\rm Symp}^2(\R^{2n}) \cong \R^{2n} \times \R^{2n(2n-1)/2}.
$$
Now we globalize the above discussion to the \emph{contactible} manifolds. 
Then a  contact structure $\xi$ on $M$ is is a smooth section of the bundle 
$$
\mathfrak{C}(M) : = \Gamma(\mathfrak{C}(TM)) \to M.
$$
This completes the discussion of space of contact structures which clearly
shows that it carries a Frechet manifold structure modeled by $(\R^{2n})^* \times \text{\rm Symp}^2(\R^{2n})$.
This completes the description of Frechet manifold structure of $\mathfrak{C}(M)$.
 
\bigskip

\def\cprime{$'$}
\providecommand{\bysame}{\leavevmode\hbox to3em{\hrulefill}\thinspace}
\providecommand{\MR}{\relax\ifhmode\unskip\space\fi MR }
\providecommand{\MRhref}[2]{%
  \href{http://www.ams.org/mathscinet-getitem?mr=#1}{#2}
}
\providecommand{\href}[2]{#2}



\end{document}